%% file: MOTZ_022022.tex
\documentclass[12pt]{article}%
\usepackage{fullpage}
\usepackage{graphicx}
\usepackage{float,enumerate}
\usepackage{makeidx}
\usepackage{graphicx}
\usepackage{amsmath}
\usepackage{amssymb}
\usepackage{amsfonts}
\usepackage{rotating}
\usepackage{comment}
\usepackage{color}
\usepackage{algorithm}
\usepackage{tikz}
\usepackage{hyperref}
\usetikzlibrary{plotmarks}
\newcommand{\trianglecolored}[1]{\protect\tikz{\protect\filldraw[draw=#1,fill=#1] (0,0) --
			(0.6em,0) -- (0.3em,0.6em);}}
\usepackage{pgfplots}
\usepgflibrary{shapes.geometric}
\usepackage{standalone}
\usetikzlibrary{patterns,angles, arrows, quotes, calc}
\usepackage{algorithmic, algorithm}
\usepackage{xcolor}
\usepackage{subcaption}

\newtheorem{theorem}{Theorem}[section]

\newtheorem{definition}[theorem]{Definition}

\newtheorem{lemma}[theorem]{Lemma}
\newtheorem{notation}[theorem]{Notation}

\newtheorem{proposition}[theorem]{Proposition}
\newtheorem{remark}[theorem]{Remark}

\newenvironment{proof}[1][Proof]{\noindent\textbf{#1.} }{\ \rule{0.5em}{0.5em}}

\newcommand{\Nin}{\mathcal{N}_{\operatorname{dof}}}
\newcommand{\Nout}{\mathcal{N}_{\operatorname{test}}}
\newcommand{\N}{\mathcal{N}}

\renewcommand{\Re}{\operatorname*{Re}}
\newcommand{\alg}{MOTZ}

\newcommand{\tabcell}[1]{\begin{tabular}{@{}c@{}}#1\end{tabular}}

\begin{document}

\title{Solvability of Discrete Helmholtz Equations}

\author{
Maximilian Bernkopf
	\thanks{\href{mailto:maximilian.bernkopf@asc.tuwien.ac.at}{maximilian.bernkopf@asc.tuwien.ac.at},
	 
		Institute
		for Analysis and Scientific Computing, TU Wien, Wiedner Hauptstr.
		8-10,
		A-1040
		Vienna, Austria}
\and
Stefan Sauter \thanks{\href{mailto:stas@math.uzh.ch}{stas@math.uzh.ch}, 
Institut f\"{u}r Mathematik,
	Universit\"{a}t Z\"{u}rich, Winterthurerstr. 190, CH-8057 Z\"{u}rich,
	Switzerland}
\and
C\'{e}line 	Torres
\thanks{\href{mailto:celine.torres@math.uzh.ch}{celine.torres@math.uzh.ch}, 
Institut f\"{u}r Mathematik,
		Universit\"{a}t Z\"{u}rich, Winterthurerstr. 190, CH-8057 Z\"{u}rich,
		Switzerland},
\and	
Alexander Veit
	\thanks{\href{mailto:alexander_veit@hms.harvard.edu}{alexander\_veit@hms.harvard.edu},
		Department of
		Biomedical Informatics, Harvard Medical School, 25 Shattuck St,
		Boston,
		MA
		02115, USA}
	}
\maketitle

\begin{abstract}
{We study the unique solvability of the discretized Helmholtz problem with
Robin boundary conditions using a conforming Galerkin $hp$-finite element
method. Well-posedness of the discrete equations is typically investigated by
applying a compact perturbation argument to the continuous Helmholtz problem so that a
\textquotedblleft sufficiently rich\textquotedblright\ discretization results
in a \textquotedblleft sufficiently small\textquotedblright\ perturbation of
the continuous problem and well-posedness is inherited via Fredholm's
alternative. The qualitative notion \textquotedblleft sufficiently
rich\textquotedblright, however,\ involves unknown constants and is only of
asymptotic nature.

Our paper is focussed on a fully discrete approach by mimicking the tools for
proving well-posedness of the continuous problem \textit{directly} on the
discrete level. In this way, a computable criterion is derived which certifies
discrete well-posedness without relying on an asymptotic perturbation
argument. By using this novel approach we obtain a) new existence and
uniqueness results for
the $hp$-FEM for the Helmholtz problem b) examples for meshes such that the
discretization becomes unstable (stiffness matrix is singular), and c) a simple
checking Algorithm MOTZ \textquotedblleft
marching-of-the-zeros\textquotedblright\ which guarantees in an a posteriori
way that a given mesh is certified for a well-posed
Helmholtz
discretization.}
{Helmholtz equation at high wave number; adaptive mesh
generation; pre-asymptotic stability; hp-finite elements; a posteriori
stability.}
\end{abstract}

\section{Introduction}

In this paper, we consider the numerical discretization of the Helmholtz
problem for modelling acoustic wave propagation in a bounded Lipschitz
domain $\Omega \subset \mathbb{R}^{d}$, $d=1,2$, with boundary $\Gamma
:=\partial \Omega $. Robin boundary conditions are imposed on $\Gamma $ and
the strong form is given by seeking \(u\) s.t.
\begin{equation*}
\begin{array}{cc}
-\Delta u-k^{2}u=f & \text{in }\Omega , \\
\dfrac{\partial u}{\partial \mathbf{n}}-\operatorname{i}ku=g & \text{on
}\Gamma .%
\end{array}%
\end{equation*}%
Here, $\mathbf{n}$ denotes the outer normal vector and $k\in \mathbb{R}%
\backslash \left\{ 0\right\} $ is the wave number. It is well known that the
weak formulation of this problem is well posed; the proof is based on
Fredholm's alternative in combination with the unique continuation principle
(u.c.p.) (see, e.g., \cite{Leis86}).
The restriction to Robin boundary conditions is only to fix the 
ideas. Our
method and theory apply verbatim to any other boundary condition of the type%
\begin{equation*}
\dfrac{\partial u}{\partial \mathbf{n}}-\operatorname{i}T_{k}u=g\quad \text{on
}%
\Gamma
\end{equation*}%
for some dissipative linear operator $T_{k}$. The
generalization to mixed
boundary condition which are imposed only on a subset of $\Gamma$
with positive surface measure is straightforward as well: only the
initialisation of Algorithm 1 (see \S \ref{sec:ducp}) has to be restricted to
the
degrees of freedom which lie on the boundary part with mixed boundary
conditions.

We consider the discretization of this equation (in variational form) by a
conforming Galerkin method. The proof of well-posedness for this
discretization goes back to \cite{Schatz74} and is based on a perturbation
argument: the subspace which defines the Galerkin discretization has to be
sufficiently \textquotedblleft rich\textquotedblright\ in the sense that a
certain adjoint approximation property holds. However, this adjoint
approximation property contains a constant which is a priori unknown. The
existing analysis gives insights into how the parameters defining the
Galerkin space should be chosen \textit{asymptotically} but does not answer
the question whether, for a concrete finite dimensional space, the
corresponding Galerkin discretization has a unique solution.

In one spatial dimension on quasi-uniform grids the condition
$kh\leq
C_{\operatorname*{res}}$, for some sufficiently small \textit{minimal
resolution constant }$C_{\operatorname*{res}}=O\left(  1\right) $, ensures
unique solvability of the Galerkin discretization. This result is proved for
linear finite elements on uniform meshes with $C_{\operatorname*{res}}=1$ in
\cite[Thm.\ 4]{Ihlenburg} and for the $hp$ version of the finite
element method
on uniform 1D meshes with $C_{\operatorname*{res}}<\pi$ in \cite[Thm.~3.3]%
{Frankp},
i.e., on uniform grids employing $hp$ finite elements the
Galerkin discretization has a unique solution if $kh \leq C_{\operatorname*{res}}<\pi$.
On the other hand, \textit{piecewise linear} finite elements in two
or more spatial dimensions on a star-shaped domain with smooth boundary or a convex polygonal domain,
a straightforward application of the Schatz' perturbation
argument leads to the much more restrictive condition: $k^{2}h\leq
C_{\operatorname*{res}}\ $for $C_{\operatorname*{res}}$ sufficiently small.
For piecewise linear elements in spatial dimension two or three this result
can be improved: if the computational domain $\Omega$ is a convex
polygon/polyhedron, the condition \textquotedblleft$k^{3}h^{2}\leq
C_{\operatorname*{res}}$ for sufficiently small $C_{\operatorname*{res}%
}=O\left(  1\right)  $\textquotedblright\ ensures
existence and uniqueness (see
\cite[Thm.~6.1]{wu14}). For the $hp$-finite element method on an analytically bounded or convex polygonal domain,
the analysis in
\cite{MelenkSauterMathComp}, \cite{mm_stas_helm2} leads to the condition
\textquotedblleft$kh/p\leq C_{\operatorname*{res}}$ for
$C_{\operatorname*{res}}$ sufficiently small provided the polynomial degree
satisfies $p\gtrsim\log k$\textquotedblright.

Since no sharp bounds for the resolution constant $C_{\operatorname*{res}}$
are available for general conforming finite element meshes in 2D and 3D such
estimates have merely qualitative and asymptotic character. This drawback was
the motivation for the development of many novel discretization techniques by
either modifying the original sesquilinear form or employing a discontinuous
Galerkin discretization. Such discretizations have in common that unique
solvability of the discrete problem does not rely on the Schatz argument.
Unique solvability can therefore be established under less
restrictive
conditions; we mention \cite{burmann-wu-zhu16, wu14} for an analysis of a
continuous
interior penalty method with piecewise linear elements.
In \cite[Cor.~3.5]{wu14} unique solvability is established
for general polygonal/polyhedral domains for any $k>0$ and $h>0$.
In \cite{feng-wu11} interior penalty $hp$-DG methods are analysed and
unique solvability for polygonal/polyhedral star-shaped domains is shown
for any $k>0$ and $h>0$ under certain conditions on the penalty parameters,
see \cite[Thm~3.2]{feng-wu11} for details.
Finally, in \cite{chen-qiu17} a least-squares approach is analysed,
establishing unique solvability for domains, for which a-priori
estimates of the continuous problem are available, see
\cite[Thm~2.4]{chen-qiu17}.

We do not go into the details of such methods because our focus in this paper
is on the question whether the conforming Galerkin discretization of the
Helmholtz problem with Robin boundary conditions can lead to a system matrix
which is singular and how to define a computable criterion to guarantee that
for a given mesh the conforming Galerkin discretization is unique. Such an
approach can be also viewed as a \textit{novel a posteriori strategy}: the
goal is not to improve accuracy but to guarantee unique solvability without
relying on a resolution condition for a conforming Galerkin discretization.

The given finite element mesh is the input of our new algorithm called MOTZ
(\textquotedblleft marching-of-the-zeros\textquotedblright{}) which analyses
the mesh, based on a stability criterion which we will develop in this paper.
If the result is \textquotedblleft certified\textquotedblright\ then the
piecewise linear, conforming finite element discretization of the Helmholtz
problem with Robin boundary conditions leads to a regular system matrix.
Otherwise, the triangulation ist \textit{marked} by MOTZ as
\textquotedblleft critical\textquotedblright\ and we will present a \textit{
	local mesh modification algorithm} \textquotedblleft
MOTZ\_flip\textquotedblright\ with the \textit{goal} to obtain a modified mesh
which is certified by MOTZ. In this way, MOTZ can be regarded as
an a posteriori\textit{ stability indicator}. In contrast there
exist a
posteriori\textit{ error estimators} for the Helmholtz problem in
the literature \cite{doerfler_sauter}, \cite{sauter_zech},
\cite{Irimie_Helm_apost}, \cite{chaumontfrelet_apost} which take as input a
computed discrete solution and estimates the arising error in order to mark
(in an ideal situation) those elements which contain the largest error
contributions. However, all these error estimators are fully reliable and
efficient only if a resolution condition is satisfied and the discrete system
is well posed. They fail if the discrete solution does not exist and differ
from our approach with respect to their goal (approximability in contrast to
well-posedness). Also the adaptive algorithm in \cite{Bespalov2019} is
essentially of this \textit{error estimator} type
-- however has the feature that the mesh is refined uniformly if the discrete
solution does not exist. To the best of our knowledge, our approach is the
first one which does not assume such conditions to hold and refines adaptively
for the goal to improve stability.

The paper is organized as follows. In Section \ref{SecSetting}, we formulate
the Helmholtz problem in a variational setting and recall the relevant
existence and uniqueness results. Then, the conforming Galerkin discretization
by $hp$-finite elements is introduced; by using a standard nodal basis the
discrete problem is formulated as a matrix equation of the form $\mathbf{A}%
_{k}\mathbf{u}=\mathbf{r}$. Since the resulting system is finite dimensional
it is sufficient to prove uniqueness of the homogeneous problem in order to
get discrete solvability. We formulate this condition and obtain in a
straightforward manner that the discrete homogeneous solution must vanish on
the boundary $\Gamma$.

In Section \ref{KkSec} we discuss the invertibility of $\mathbf{A}_{k}$ for
different scenarios. First, we prove in the one-dimensional case, i.e., $d=1$,
that $\mathbf{A}_{k}$ is regular for any conforming $hp$-finite element space
without any restrictions on the mesh size and the polynomial degree $p$. In
contrast, we show in Section \ref{subsection:singular_2d_example} that the
matrix $\mathbf{A}_{k}$ can become singular for two-dimensional domains at
certain discrete wave numbers for simplicial/quadrilateral meshes. We present
an example of a regular triangulation of the square domain $\left(
-1,1\right)  ^{2}$ such that the conforming piecewise linear finite element
discretization of the Helmholtz problem with Robin boundary conditions leads
to a singular system matrix $\mathbf{A}_{k}$. This generalizes the example in
\cite[Ex. 3.7]{MPS13} where the finite element space satisfies homogeneous
Dirichlet boundary conditions to the case of Robin boundary conditions. Next,
we discuss conforming finite element discretizations on quadrilateral meshes
and show that for rectangular domains and tensor product quadrilateral meshes
the matrix $\mathbf{A}_{k}$ is always invertible for polynomial degree
$p\in\left\{  1,2,3\right\}  $ by applying a local argument inductively. We
also show that there are mesh configurations where this local argument breaks
down for $p=4$.

Motivated by the results in Section \ref{subsection:singular_2d_example}, we
present in Section \ref{sec:ducp} the aforementioned algorithm MOTZ. For the
case that the outcome is \textquotedblleft{}critical\textquotedblright{} we
also
present two companion algorithms which refine or modify the given mesh such
that the MOTZ algorithm returns \textquotedblleft{}certified\textquotedblright.
The section is complemented by numerical experiments and a short discussion on
the behaviour of the inf-sup constant before and after the
modification of a ``critical'' mesh to a ``certified'' mesh.

\section{Setting\label{SecSetting}}

Let $\Omega\subset\mathbb{R}^{d}$, $d=1,2$ be a bounded Lipschitz domain with
boundary $\Gamma:=\partial\Omega$.
Let $L^{2}\left(  \Omega\right)  $ denote the usual Lebesgue space with scalar
product denoted by $\left(  \cdot,\cdot\right)  $ (complex conjugation is on
the second argument) and norm $\Vert\cdot\Vert_{L^{2}(\Omega)}:=\left\Vert
\cdot\right\Vert :=\left(  \cdot,\cdot\right)  ^{1/2}$. Let $H^{1}\left(
\Omega\right)  $ denote the usual Sobolev space and let $\gamma:H^{1}\left(
\Omega\right)  \rightarrow H^{1/2}\left(  \Gamma\right)  $ be the standard
trace operator. We introduce the sesquilinear forms
\[
a_{0,k}\left(  u,v\right)  :=\left(  \nabla u,\nabla v\right)  -k^{2}\left(
u,v\right)  \qquad\forall u,v\in H^{1}\left(  \Omega\right)
\]
and
\[
b_{k}\left(  u, v\right)  :=-\operatorname*{i}k\langle
u, v\rangle _{\Gamma}\qquad\forall u,v\in H^{1/2}\left( \Gamma\right)  ,
\]
where $\langle \cdot,\cdot\rangle_{\Gamma}$ denotes the $L^{2}$ scalar
product on the boundary $\Gamma$. Throughout this paper we assume that the
wave number $k$ satisfies%
\[
k\in\overset{\circ}{\mathbb{R}}:=\mathbb{R}\backslash\left\{  0\right\}
\text{.}%
\]
The weak formulation of the Helmholtz problem with Robin boundary conditions
is given as follows: For \(f\in L^2(\Omega)\) and \(g \in
H^{1/2}(\Gamma)\), we define $F=(f,\cdot)+\langle
g,\gamma\ \cdot\rangle_{\Gamma}\in(H^{1}\left(
\Omega\right)  )^{\prime}$. We seek $u\in H^{1}\left(  \Omega\right)  $ such
that
\begin{equation}
a_{k}\left(  u,v\right)  :=a_{0,k}\left(  u,v\right)  +b_{k}\left(  \gamma
u,\gamma v\right)  =F\left(  v\right)  \quad\forall v\in H^{1}\left(
\Omega\right)  .\label{varformrobin}%
\end{equation}
In the following, we will omit the trace operator $\gamma$ in the notation of
the sesquilinear form since this is clear from the context. Well-posedness of problem (\ref{varformrobin}) is proved in \cite[Prop.
8.1.3]{MelenkDiss}.

\begin{proposition}
Let $\Omega$ be a bounded Lipschitz domain. Then, (\ref{varformrobin}) is
uniquely solvable for all $F\in\left(  H^{1}\left(  \Omega\right)  \right)
^{\prime}$ and the solution depends continuously on the data.
\end{proposition}

We employ the conforming Galerkin finite element method for its discretization
(see, e.g., \cite{Ciarlet_orig}, \cite{scottbrenner3}). For the spatial
dimension, we assume
that $\Omega
\subset\mathbb{R}^{d}$ is an interval for $d=1$ or a polygonal domain for
$d=2$. We consider either conforming meshes \(\mathcal{K}_{\mathcal{T}}\)
(i.e., no hanging nodes) composed of
closed simplices or conforming meshes \(\mathcal{K}_{\mathcal{Q}}\)
composed of
quadrilaterals. The set of all
vertices is denoted by
$\mathcal{N}$, i.e., for $\mathcal{M}\in\left\{  \mathcal{T},\mathcal{Q}%
\right\}  $
\[
\mathcal{N}:=\left\{x\in \mathbb{R}^d~\mid\exists
K\in\mathcal{K}_{\mathcal{M}}\text{ with
}x\text{
is a vertex of }K\right\}.
\]
For $d=2$, we denote the set of all edges by $\mathcal{E}$ and
\[
\mathcal{E}_{\Omega}:=\left\{  E\in\mathcal{E}:E\not \subset \Gamma \right\}  .
\]
For $p\in\mathbb{N}$, we define the continuous, piecewise polynomial finite
element space by%
\begin{align*}
&  \text{ for }d\in\{1,2\} &  &  S_{\mathcal{T}}^{p}:=\left\{  u\in
C^{0}\left(  \Omega\right)  \mid\forall K\in\mathcal{K}_{\mathcal{T}}:u_{|{K}%
}\in\mathcal{P}^{p}\right\}  ,\\
&  \text{ for }d=2 &  &  S_{\mathcal{Q}}^{p}:=\left\{  u\in C^{0}\left(
\Omega\right)  \mid\forall K\in\mathcal{K}_{\mathcal{Q}}:u_{|{K}}%
\in\mathcal{Q}^{p}\right\}  ,
\end{align*}
where $\mathcal{P}^{p}$ (resp. $\mathcal{Q}^{p}$) is the space of multivariate
polynomials of maximal total degree $p$ (resp. maximal degree $p$ with respect
to each variable). The reference elements are given by%
\[
\widehat{K}_{\mathcal{T}}:=\left\{  \mathbf{x}=\left(  x_{i}\right)
_{i=1}^{d}\in\mathbb{R}_{\geq0}^{d}:\sum_{i=1}^{d}x_{i}\leq1\right\}  \text{
for }d\in\{1,2\},\qquad\widehat{K}_{\mathcal{Q}}:=[-1,1]^{2}.
\]
For $\mathcal{M}\in\left\{  \mathcal{T},\mathcal{Q}\right\}  $ and
$K\in\mathcal{K}_{\mathcal{M}}$, let $\phi_{K}:\widehat{K}_{\mathcal{M}%
}\rightarrow K$ denote an affine pullback. Moreover for $p\geq1$, we denote by
$\hat{\Sigma}^{p}$ a set of nodal points in $\widehat{K}_{\mathcal{M}}$
unisolvent on the corresponding polynomial space which allow to impose
continuity across faces. The nodal points on $K\in\mathcal{K}_{\mathcal{M}}$
are then given by lifting those of the reference element:%
\[
\Sigma_{K}^{p}:=\left\{  \phi_{K}\left(  z\right)  :z\in\hat{\Sigma}%
^{p}\right\}  .
\]
The set of global nodal points is given by%
\[
\Sigma^{p}:={\bigcup\nolimits_{K\in\mathcal{K}_{\mathcal{M}}}}\Sigma_{K}^{p},
\]
and we denote by $\left(  b_{z,p}\right)  _{z\in\Sigma^{p}}\subset
S_{\mathcal{M}}^{p}$ the standard Lagrangian basis.
\begin{remark}
\label{rmk:notation}For $\mathcal{M}\in\left\{  \mathcal{T},\mathcal{Q}%
\right\}  $, we write $S$ or $S_{\mathcal{M}}$ short for $S_{\mathcal{M}}^{p}$
and $\Sigma$ short for $\Sigma^{p}$ if no confusion is possible. If $p=1$ then
the two sets $\mathcal{N}$ and $\Sigma^{1}$ are equal and we use
the notation $\mathcal{N}$ for the set of degrees of freedom.
\end{remark}
The Galerkin finite element method for the discretization of
(\ref{varformrobin}) is given by:%
\begin{equation}
\text{find }u_{S}\in S\text{ such that }a_{k}\left(  u_{S},v\right)  =F\left(
v\right)  ,\quad\forall v\in S\text{.}\label{GalDisc}%
\end{equation}
The basis representation allows us to reformulate (\ref{GalDisc}) as a linear
system of equations. The system matrices $\mathbf{K}=\left(  \alpha
_{y,z}\right)  _{y,z\in\Sigma}$, $\mathbf{M}=\left(  \mu_{y,z}\right)
_{y,z\in\Sigma}$, $\mathbf{B}=\left(  \beta_{y,z}\right)  _{y,z\in\Sigma}$ are
given by%
\[
\alpha_{y,z}:=\left(  \nabla b_{z,p},\nabla b_{y,p}\right)  ,\quad\mu
_{y,z}:=\left(  b_{z,p},b_{y,p}\right)  ,\quad\beta_{y,z}:=\langle
b_{z,p},b_{y,p}\rangle_\Gamma
\]
and the right-hand side vector $\mathbf{r}=\left(  \rho_{z}\right)
_{z\in\Sigma}$ by $\rho_{z}=F\left(  b_{z,p}\right)$. Then, the Galerkin finite
element discretization leads to the
following system of linear equations%
\begin{equation}
\mathbf{A}_{k}\mathbf{u}=\mathbf{r}\quad\text{with\quad}\mathbf{A}%
_{k}:=\mathbf{K}-k^{2}\mathbf{M-}\operatorname*{i}k\mathbf{B} \label{Kk},
\end{equation}
where \(\mathbf{u} = (u_z)_{z\in \Sigma}\).
We start off with some general remarks. It is well known that the sesquilinear
form $a_{k}\left(  \cdot,\cdot\right)  $ satisfies a G\aa rding inequality and
Fredholm's alternative tells us that well-posedness of \eqref{varformrobin}
follows from uniqueness. Similarly, the finite dimensional problem
\eqref{GalDisc} is well posed if the problem%
\begin{equation}
\text{find }u\in S\text{ such that }a_{k}\left(  u,v\right)  =0\quad\forall
v\in S \label{eq:homsol}%
\end{equation}
has only the trivial solution. We note that if we choose $v=u$ and consider the
imaginary part of \eqref{eq:homsol}, we get
\begin{equation}
0=\operatorname{Im}a_{k}\left(  u,u\right)  =-k\left\Vert u\right\Vert
_{\Gamma}^{2}\implies\left.  u\right\vert _{\Gamma}=0. \label{eq:impzero}%
\end{equation}

\section{Regularity of the discrete system matrix
$\mathbf{A}_{k}$}\label{KkSec}
This section covers three different topics concerning the discretization of
the Helmholtz equation. First, in Section \ref{subsection:1d_ucp} we analyse
the conforming Galerkin finite element method with polynomials of degree
$p\geq1$ in spatial dimension one. Next, in Section
\ref{subsection:singular_2d_example}, we present a singular two-dimensional
example for piecewise linear finite elements on a triangular mesh. Finally, in
Section \ref{subsection:quads} we consider structured tensor-product meshes in
spatial dimension two.

\subsection{The one-dimensional case\label{subsection:1d_ucp}}
We prove that the conforming Galerkin finite element method for the
one-dimensional Helmholtz equation is well posed for any $k\in\overset{\circ
}{\mathbb{R}}$.

\begin{theorem}
\label{thm:1d_ucp} Let $\Omega\subset\mathbb{R}$ be a bounded interval and
consider the Galerkin discretization (\ref{GalDisc}) of (\ref{varformrobin})
with conforming finite elements. Then, for any $k\in\overset{\circ
}{\mathbb{R}}$ the matrix $\mathbf{A}_{k}$ in (\ref{Kk}) is regular.
\end{theorem}
\begin{proof}
We assume that $\Omega=\left(  -1,1\right)  $ (the result for general
intervals follows by an affine transformation).
Since the problem (\ref{GalDisc}) is linear and finite dimensional it suffices
to show that the only solution of the homogeneous problem (\ref{eq:homsol}) is
$u=0$ . From (\ref{eq:impzero}) we already know that $u\left(  -1\right)
=u\left(  1\right)  =0$. We assume that the intervals $\mathcal{K}%
_{\mathcal{T}}=\left\{  K_{i}:1\leq i\leq N\right\}  $ are numbered from left
to right, so that $K_{N}=\left[  x_{N-1},1\right]  $. The function
$u_{N}:=\left.  u\right\vert _{K_{N}}$ can be written as%
\[
u_{N}=\sum_{z\in\Sigma_{K_{N}}^{p}\backslash\left\{  1\right\}  }u_{z}%
b_{z,p}.
\]
As test functions, we employ the functions $b_{y,p}$, $y\in\Sigma_{K_{N}}%
^{p}\backslash\left\{  x_{N-1}\right\}  $ and obtain from (\ref{eq:homsol})
\[
\sum_{z\in\Sigma_{K_{N}}^{p}\backslash\left\{  1\right\}  }\left(
\alpha_{y,z}-k^{2}\mu_{y,z}\right)  u_{z}=0\quad\forall y\in\Sigma_{K_{N}}%
^{p}\backslash\left\{  x_{N-1}\right\}  .
\]
This is a $p\times p$ linear system and our goal is to show that it is regular
for all $k\in\overset{\circ}{\mathbb{R}}$ so that $u_{z}=0$ follows for all
$z\in\Sigma_{K_{N}}^{p}\backslash\left\{  1\right\}  $. By an affine
transformation this is equivalent to the following implication
\begin{equation}
\forall \tilde{k}\in\overset{\circ}{\mathbb{R}}~\text{\textquotedblleft find }%
u\in\mathcal{P}_{0)}^{p}\text{ such that }\left(  u^{\prime},v^{\prime
}\right)  -\tilde{k}^{2}\left(  u,v\right)  =0\quad\forall
v\in\mathcal{P}_{(0}%
^{p}\text{\textquotedblright}\implies u=0, \label{uunit}%
\end{equation}
where \(\tilde{k}= 2|K_N|k\) and
\[
\mathcal{P}_{(0}^{p}:=\left\{  v\in\mathcal{P}^{p}\mid v\left(  -1\right)
=0\right\}  \quad\text{and\quad}\mathcal{P}_{0)}^{p}:=\left\{  v\in
\mathcal{P}^{p}\mid v\left(  1\right)  =0\right\}  .
\]
Let $u\in\mathcal{P}^{p}_{0)}$ satisfy the assumption in \eqref{uunit}. For
$\omega_{(0}\left(
x\right)  :=1+x$ we choose $v=\omega_{(0}u^{\prime}\in\mathcal{P}^{p}_{(0}$ as
a test function. We integrate by parts, use the fact
$2\Re(u\overline{u}^{\prime})=(|u|^{2})^{\prime}$,
and employ the endpoint properties of $u$ and $v$ to
obtain%
\begin{align*}
0  &  =\Re\left[  \left(  u^{\prime},\left(  \omega_{(0}u^{\prime}\right)
^{\prime}\right)  -\tilde{k}^{2}\left(  u,\omega_{(0}u^{\prime}\right)  \right]
\\
&  =-\Re\left[  \left(  u^{\prime\prime},\omega_{(0}u^{\prime}\right)
\right]  +\left.  \omega_{(0}\left\vert u^{\prime}\right\vert ^{2}\right\vert
_{-1}^{1}-\tilde{k}^{2}\left(  \frac{1}{2}\left(  \left\vert u\right\vert
^{2}\right)
^{\prime},\omega_{(0}\right) \\
&  =-\left(  \frac{1}{2}\left(  \left\vert u^{\prime}\right\vert ^{2}\right)
^{\prime},\omega_{(0}\right)  +2\left\vert u^{\prime}\left(  1\right)
\right\vert ^{2}+\frac{\tilde{k}^{2}}{2}\left\Vert u\right\Vert ^{2}-\left.
\frac{\tilde{k}^{2}}{2}\left\vert u\right\vert ^{2}\omega_{(0}\right\vert
_{-1}^{1}\\
&  =\frac{1}{2}\left\Vert u^{\prime}\right\Vert ^{2}+\left\vert u^{\prime
}\left(  1\right)  \right\vert ^{2}+\frac{\tilde{k}^{2}}{2}\left\Vert
u\right\Vert
^{2}.
\end{align*}
This implies that $u_{N}=\left.  u\right\vert _{K_{N}}=0$ and we may proceed
to the adjacent interval $K_{N-1}$. Since $u\left(  x_{N-1}\right)  =0$ we
argue as before to obtain $\left.  u\right\vert _{K_{N-1}}=0$. The result
follows by induction.
\end{proof}

\subsection{A singular example in two dimensions}

\label{subsection:singular_2d_example} In this section, we will present an
example which illustrates that Robin boundary conditions are not sufficient to
ensure well-posedness of the Galerkin discretization (although the continuous
problem is well posed). For this, we first introduce the definition of a
weakly acute angle condition on a triangulation.
\begin{definition} Let $d=2$ and \(\mathcal{T}\) a conforming
triangulation of
the domain \(\Omega\). For an edge \(E\in\mathcal{E}_\Omega\), we
	denote by $\tau_{-}$, $\tau_{+}$ the two triangles sharing $E$. Let
	$\alpha_{-}$ and $\alpha_{+}$ denote the angles in $\tau_{-}$, $\tau_{+}$
	which are opposite to $E$. We define the angle
	\[
	\alpha_{E}:=\alpha_{-}+\alpha_{+}.
	\]
	We say that the edge \(E\) satisfies the weakly acute angle condition if
	\(\alpha_E\leq 	\pi\).
\end{definition}
The following observation is key for the proof of the upcoming Lemma
\ref{lem:unstablemesh}, as
well
as for the derivation of the Algorithm MOTZ (\textquotedblleft
marching-of-the-zeros\textquotedblright{}) presented in Section 4.
\begin{lemma} \label{lem:anglecond}
Let \(p=1\), i.e. \(S= S^1_{\mathcal{T}}\). Let
\(E\in\mathcal{E}_\Omega\) be an
inner edge connecting two
	nodes \(z,z^\prime\).
	Then
	\[\left(  \nabla b_{z},\nabla b_{z^{\prime}}\right)  \leq0 \quad \text{if
	and only if}\quad \alpha_{E}\leq\pi\ \text{ (i.e., } E \text{ satisfies
	the weakly acute angle condition)}.\]
\end{lemma}
\begin{proof}
	This follows, e.g., from \cite[(6.8.7) Satz]{SauterDipl} (see, e.g.,
	\cite[p.
	78]{StrangFix} for a sufficient criterion).
\end{proof}

For our example, we consider $\Omega=\left(  -1,1\right)  ^{2}$ and a
mesh as depicted in Figure \ref{figure:mesh_counter_example_trigs}. We employ
globally continuous piecewise linear finite elements, i.e., $S=S_{\mathcal{T}%
}^{1}$. The degrees of freedom on the boundary $\Gamma$ are located at
$P_{1}^{\Gamma}=(-1,-1)^{T}$, $P_{2}^{\Gamma}=(1,-1)^{T}$, $P_{3}^{\Gamma
}=(1,1)^{T}$ and $P_{4}^{\Gamma}=(-1,1)^{T}$. The inner degrees of freedom are
located at $P_{1}^{\Omega}=(-\alpha,0)^{T}$, $P_{2}^{\Omega}=(0,-\alpha)^{T}$,
$P_{3}^{\Omega}=(\alpha,0)^{T}$, $P_{4}^{\Omega}=(0,\alpha)^{T}$ and
$P_{5}^{\Omega}=(0,0)^{T}$, with parameter $\alpha\in(0,1)$. The mesh is
denoted by $\mathcal{T}_{\alpha}$.
We denote the unknowns as $u_{i}^{\Gamma}$ and $u_{i}^{\Omega}$ as well as the
associated basis functions with $b_{i}^{\Gamma}$ and $b_{i}^{\Omega}$ respectively.

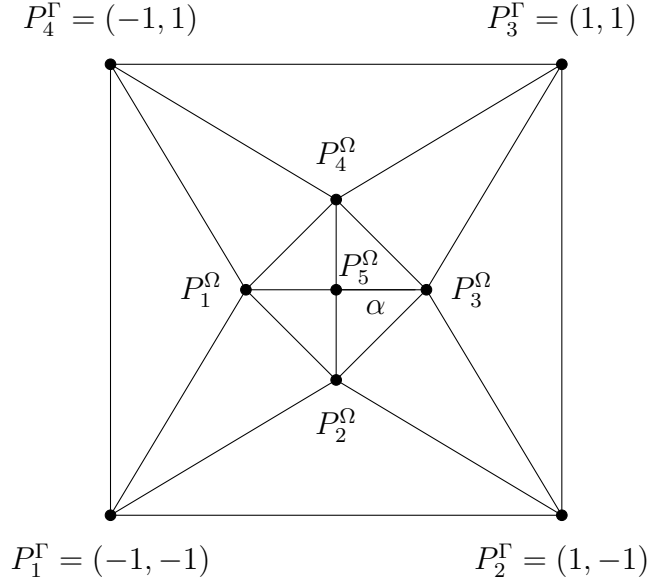
\begin{figure}[ptb]
	\centering
	\input{supplements/counter_example_trigs_1.tex}
	\caption{Mesh resulting in a non-trivial solution of the Helmholtz
	equation.}
	\label{figure:mesh_counter_example_trigs}
\end{figure}

\begin{lemma}\label{lem:unstablemesh}
Let $\alpha \in (0,1)$ and define $k_c$ as
\begin{equation}\label{eq:critical_wavenumber}
	k_c = \sqrt{\frac{6(2-\alpha)}{\alpha (1-\alpha)}}.
\end{equation}
Then for any $k\in\overset{\circ}{\mathbb{R}}$
the corresponding Galerkin discretization \eqref{GalDisc}
of \eqref{varformrobin} with conforming
piecewise linear elements on the mesh $\mathcal{T}_{\alpha}$ is well posed if $k \neq \pm k_c$.
For $k = \pm k_c$ the system matrix is singular and its kernel
has dimension one.
\end{lemma}

\begin{proof}
We construct an explicit non-trivial solution $u_{h}\in S_{\mathcal{T}}^{1}$
to the homogeneous equations. To that end, note that by \eqref{eq:impzero} we
have $u_{i}^{\Gamma}=0$ for $i=1,\dots,4$. We seek a non-trivial solution to
\begin{equation}
a_{0,k}(u_{h},v_{h})=(\nabla u_{h},\nabla v_{h})-k^{2}(u_{h},v_{h}%
)=0\quad\forall v_{h}\in S_{\mathcal{T}}^{1}.\label{eq:homogeneous_equations}%
\end{equation}
Our strategy is the following: We first test with the degrees of freedom
associated to the boundary. This
allows us to construct a candidate for a non-trivial solution. Next, we
test with the interior degrees of freedom. This allows to show the existence of a
critical wave number $k_c$ as stated in the present lemma,
which can be explicitly calculated.
Furthermore, we verfiy that the kernel of the system matrix for $k=\pm k_c$ is
in fact one dimensional. Finally, we show that for any other $k \neq \pm k_c$ 
the system matrix is regular.

We test with the
hat functions $b_{i}^{\Gamma}$ for $i=1,\dots,4$ associated to the degrees of
freedom on the boundary in \eqref{eq:homogeneous_equations}. Due to their
support, they do not interact with the hat function $b_{5}^{\Omega}$. We start
with the hat function $b_{1}^{\Gamma}$. For the construction of a candidate for a non-trivial solution,
the interactions with $b_{2}^{\Gamma}$
and $b_{4}^{\Gamma}$ are redundant, since $u_{2}^{\Gamma}=u_{4}^{\Gamma}=0$.
We are therefore left with the interactions with $b_{1}^{\Omega}$ and
$b_{2}^{\Omega}$. Due to the weakly acute angle condition and Lemma
\ref{lem:anglecond}, we find that $a_{0,k}(b_{1}^{\Gamma},b_{1}^{\Omega})<0$ for
all $k \in \mathbb{R} \setminus \{ 0 \}$. Regarding $b_{2}^{\Omega}$, we find due to the symmetry of the mesh
that $a_{0,k}(b_{2}^{\Omega},b_{1}^{\Gamma})=a_{0,k}(b_{1}^{\Omega}%
,b_{1}^{\Gamma})$. The same argument holds true for testing with the other hat
functions associated to the boundary. Therefore, any solution $u_{h}$ to
\eqref{eq:homogeneous_equations} must
solve the system
\[%
\begin{pmatrix}
\gamma & \gamma & 0 & 0\\
0 & \gamma & \gamma & 0\\
0 & 0 & \gamma & \gamma\\
\gamma & 0 & 0 & \gamma
\end{pmatrix}%
\begin{pmatrix}
u_{1}^{\Omega}\\
u_{2}^{\Omega}\\
u_{3}^{\Omega}\\
u_{4}^{\Omega}%
\end{pmatrix}
=%
\begin{pmatrix}
0\\
0\\
0\\
0
\end{pmatrix}
,
\]
with $\gamma:=-a_{0,k}(b_{1}^{\Omega},b_{1}^{\Gamma})$. This system is
satisfied by $(u_{1}^{\Omega},u_{2}^{\Omega},u_{3}^{\Omega},u_{4}^{\Omega
})=(1,-1,1,-1)$. We now test with the hat function $b_{1}^{\Omega}$, which
interacts with itself, $b_{2}^{\Omega}$ and $b_{4}^{\Omega}$ as well as
$b_{5}^{\Omega}$. For some constants $a=a(\alpha)>0$ and $b=b(\alpha)>0$ we
have
\[
a_{0,k}(b_{1}^{\Omega},b_{1}^{\Omega})=a-k^{2}b.
\]
It is easy to verify that the edge \([P_1^\Omega, P_2^\Omega]\) satisfies
the weakly acute angle condition for any \(0<\alpha<1\). Therefore,
\[
a_{0,k}(b_{2}^{\Omega},b_{1}^{\Omega})=-c-k^{2}d,
\]
for some $c=c(\alpha)>0$ and $d=d(\alpha)>0$. Due to symmetry we find
\[
a_{0,k}(b_{2}^{\Omega},b_{1}^{\Omega})=a_{0,k}(b_{4}^{\Omega},b_{1}^{\Omega}).
\]
The same arguments hold true for the other test functions $b_{2}^{\Omega}$,
$b_{3}^{\Omega}$ and $b_{4}^{\Omega}$, yielding the same values. Regarding the
test function $b_{5}^{\Omega}$, the symmetry of the mesh and the satisfied
weakly
acute angle conditions imply
\[
0 > a_{0,k}(b_{i}^{\Omega},b_{5}^{\Omega})=-e\qquad\forall i=1,\dots,4
\]
for some $e=e(\alpha)>0$. Furthermore, let $f:=a_{0,k}(b_{5}^{\Omega}%
,b_{5}^{\Omega})$. The vector $(u_{1}^{\Omega},u_{2}^{\Omega},u_{3}^{\Omega
},u_{4}^{\Omega},u_{5}^{\Omega})=(1,-1,1,-1,0)$ now satisfies the following
system of equations:
\begin{equation}%
\begin{pmatrix}
a-k^{2}b & -c-k^{2}d & 0 & -c-k^{2}d & -e\\
-c-k^{2}d & a-k^{2}b & -c-k^{2}d & 0 & -e\\
0 & -c-k^{2}d & a-k^{2}b & -c-k^{2}d & -e\\
-c-k^{2}d & 0 & -c-k^{2}d & a-k^{2}b & -e\\
-e & -e & -e & -e & f
\end{pmatrix}%
\begin{pmatrix}
1\\
-1\\
1\\
-1\\
0
\end{pmatrix}
=%
\left(a+2c-k^{2}(b-2d)\right)
\begin{pmatrix}
1  \\
-1 \\
1  \\
-1 \\
0
\end{pmatrix}
.\label{singexconstr}%
\end{equation}
Below we will show $b-2d > 0$ for any $\alpha \in (0,1)$.
This allows to construct exactly two solutions $k \in \mathbb{R} \setminus \{ 0 \}$ such that $a+2c-k^{2}(b-2d) = 0$,
which in turn lets the right-hand side in
(\ref{singexconstr}) vanish so that the vector $(1,-1,1,-1,0)^{T}$ is a
solution of the homogeneous equations.

To that end, let $U$ denote the upper quadrilateral with corners
$P_4^\Gamma$, $P_4^\Omega$, $P_5^\Omega$, $P_1^\Omega$ and let
$B$ denote the bottom quadrilateral with corners
$P_1^\Gamma$, $P_1^\Omega$, $P_5^\Omega$, $P_2^\Omega$.
Furthermore, let $L$ denote the left triangle with
corners $P_1^\Gamma$, $P_4^\Gamma$, $P_1^\Omega$.
Due to the support properties of $b_1^\Omega$ and $b_2^\Omega$ we find with the above
notation that
\begin{equation*}
	b = (b_1^\Omega,b_1^\Omega)_{L^2(\Omega)} = \int_{ L \cup U \cup B } (b_1^\Omega) ^2 \, dx
	\quad \text{ and } \quad
	d = (b_1^\Omega,b_2^\Omega)_{L^2(\Omega)} = \int_{ B } b_1^\Omega b_2^\Omega \, dx
\end{equation*}
Hence, we find
\begin{align*}
	b - 2d
	&=
	\int_{ L \cup U \cup B } (b_1^\Omega) ^2 \, dx
	- 2 \int_{ B } b_1^\Omega b_2^\Omega \, dx \\
	&=
	\int_{ L } (b_1^\Omega) ^2 \, dx+
	\int_{ U } (b_1^\Omega) ^2 \, dx+
	\int_{ B } (b_1^\Omega) ^2 \, dx
	- 2 \int_{ B } b_1^\Omega b_2^\Omega \, dx \\
	&=
	\int_{ L } (b_1^\Omega) ^2 \, dx+
	\int_{ U } (b_1^\Omega) ^2 \, dx+
	\int_{ B } (b_1^\Omega - b_2^\Omega) ^2 \, dx
	- \int_{ B } (b_2^\Omega) ^2 \, dx \\
	&=
	\int_{ L } (b_1^\Omega) ^2 \, dx+
	\int_{ B } (b_1^\Omega - b_2^\Omega) ^2 \, dx > 0,
\end{align*}
where, in the last equation,
we used the fact that $\int_{ U } (b_1^\Omega) ^2 \, dx = \int_{ B } (b_2^\Omega) ^2 \, dx$,
which holds again due to the symmetry of the grid, which proves $b-2d > 0$.
In fact, tedious but elementary calculations
yield that the $\alpha$ dependent quantities $a$,  $c$, and $b-2d$
are given by
\begin{equation*}
	a = 1 + \frac{\alpha^2-2\alpha+2}{\alpha(2-\alpha)} + \frac{1}{1-\alpha}, \qquad c = \frac{1-\alpha}{\alpha(2-\alpha)}
\end{equation*}
as well as
\begin{align*}
	b - 2d
	&=
	\int_{ L } (b_1^\Omega) ^2 \, dx+
	\int_{ B } (b_1^\Omega - b_2^\Omega) ^2 \, dx \\
	&= \frac{1-\alpha}{6} + \frac{\alpha^2+(2-\alpha) \alpha}{12} = \frac{1}{6},
\end{align*}
which yields equation \eqref{eq:critical_wavenumber} for the critical wave number $k_c$, via
$a+2c-k^{2}(b-2d) = 0$.

It is left to show that the vector $(1,-1,1,-1,0)^{T}$ is in fact (up to scaling)
the only non-trivial solution to the homogeneous equations for $k \in \mathbb{R} \setminus \{ 0 \}$.
By the above arguments, any other candidate has to be of the form
$(1,-1,1,-1,\mu)^{T}$ or $(0,0,0,0,\mu)^{T}$ for some $\mu \neq 0$.

We first show that $(0,0,0,0,\mu)^{T}$ for some $\mu \neq 0$
can never be a solution to the homogeneous equations:
Assume the contrary, then similarly as in equation \eqref{singexconstr},
we find $\mu (-e, -e, -e, -e, f)^T$ has to be the zero vector for some $\mu \neq 0$,
which is impossible, since $e > 0$ for any $k \in \mathbb{R} \setminus \{ 0 \}$.

To finish the poof, we finally show that $(1,-1,1,-1,\mu)^{T}$ for some $\mu \neq 0$
can never be a solution to the homogeneous equations for $k \in \mathbb{R} \setminus \{ 0 \}$.
Again, similarly as in equation \eqref{singexconstr}, we find that the
following equations
have to be satisfied:
\begin{equation*}%
\left(a+2c-k^{2}(b-2d)\right)
\begin{pmatrix}
1  \\
-1 \\
1  \\
-1 \\
0
\end{pmatrix}
+
\mu
\begin{pmatrix}
-e \\
-e \\
-e \\
-e \\
f
\end{pmatrix}
=
\begin{pmatrix}
0 \\
0 \\
0 \\
0 \\
0
\end{pmatrix}
.%
\end{equation*}
Hence, we find that $\left(a+2c-k^{2}(b-2d)\right) - \mu e = 0$ (first equation) as well as
$-\left(a+2c-k^{2}(b-2d)\right) - \mu e = 0$ (second equation), which is only possible for $\mu = 0$, since $e > 0$ for any $k \in \mathbb{R} \setminus \{ 0 \}$.
\end{proof}
\begin{remark}
The above example shows that there exist meshes in spatial dimension two, for
which unique solvability of the discretized equations does not hold. The
constructed solution has an oscillatory behaviour which can not be ruled out by
the boundary hat functions. More generally, the same arguments hold true if
one chooses a regular $2n$ polygon, with another rotated one inside, analogous
to the above example.
\end{remark}

{
\begin{remark}[On the magnitude of the critical wave number $k$ in
Lemma~\ref{lem:unstablemesh}]\label{remark:magnitude_of_crit_wavenumber}
Lemma~\ref{lem:unstablemesh}
allows to quantify the magnitude of the critical $k_c$ for which
a non-trivial solution exists.
By varying $\alpha$ in \eqref{eq:critical_wavenumber}, the minimal $k_c$ is given
by $k_c = \sqrt{6(3+2\sqrt{2})} \approx 5.91$ for the choice $\alpha = 2-\sqrt{2}$.
In Figure~\ref{fig:alpha_vs_k_plots} (left) we visualize the behaviour
of $k_c$ in dependence of $\alpha$.
Furthermore, we present the reciprocal of the discrete inf-sup constant (see Section~\ref{SecMeshEnrichment} for further detail)
for three different values of $\alpha$, see Figure~\ref{fig:alpha_vs_k_plots}
(right).
The singularities corresponding to the critical wave number $k_c$ can be
observed
in the plot.

\begin{figure}[h]
	\begin{subfigure}{.5\textwidth}
		\begin{center}
        \input{graphics/alpha_vs_k.pgf}
    \end{center}
	\end{subfigure}%
	\begin{subfigure}{.5\textwidth}
		\begin{center}
			\input{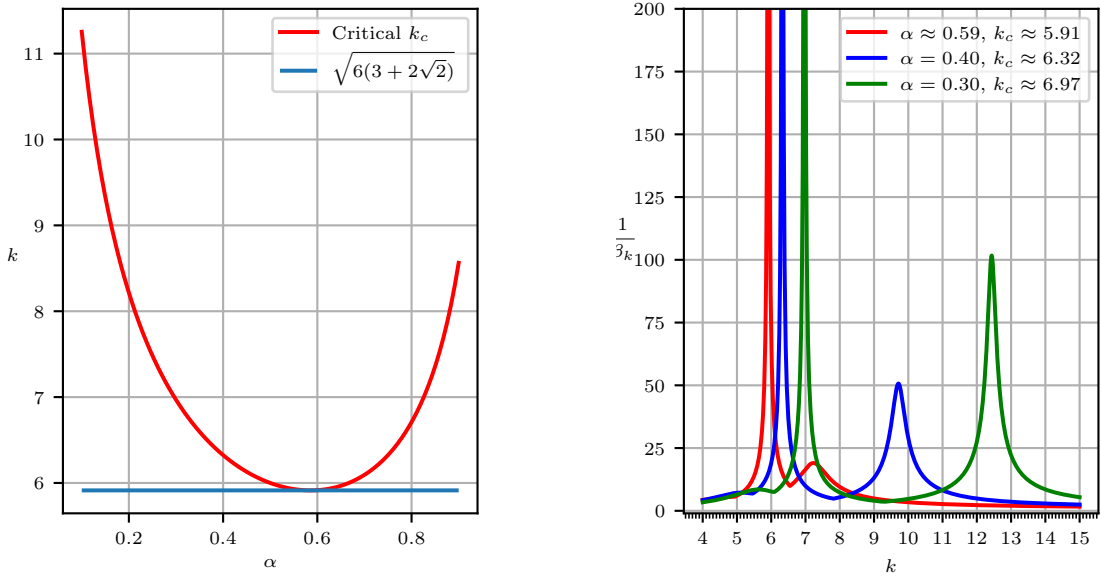}
    \end{center}
	\end{subfigure}
	\caption{Plot of $\alpha$ and the critical wave number $k_c$ (left).
	Plot of $k$ against the reciprocal of the discrete inf-sup constant for
	different values of $\alpha$ (right).}
	\label{fig:alpha_vs_k_plots}
\end{figure}
\end{remark}
}

{
For $N \in \mathbb{N}$ consider a mesh constructed by $N^2$
scaled versions of the mesh considered in Lemma~\ref{lem:unstablemesh}
as depicted in Figure \ref{figure:makromesh_singular}.
Let $\mathcal{T}_{\mathrm{macro}}$ denote the corresponding mesh.
The mesh size $h$ is then given by $h=2N^{-1}$.
A simple scaling argument together with Lemma~\ref{lem:unstablemesh}
and Remark~\ref{remark:magnitude_of_crit_wavenumber} allows to construct
non-trivial solutions the corresponding Galerkin
discretization \eqref{GalDisc} of \eqref{varformrobin} with conforming
piecewise linear elements:
\begin{figure}[ptb]
	\centering
	\input{supplements/counter_example_trigs_makro.tex}
	\caption{Mesh $\mathcal{T}_{\mathrm{macro}}$.}
	\label{figure:makromesh_singular}
\end{figure}
\begin{lemma}\label{lem:unstable_macro_mesh}
Fix $\alpha \in (0,1)$. Let $\hat{k}_{c}$ denote the critical wave number
as in equation~\eqref{eq:critical_wavenumber}.
Consider a conforming Galerkin
discretization \eqref{GalDisc} of \eqref{varformrobin} with
piecewise linear elements on the mesh $\mathcal{T}_{\mathrm{macro}}$.
Then, if $kh = 2\hat{k}_{c}$ holds true, the Galerkin discretization
is not uniquely solvable.
\end{lemma}
\begin{proof}
A scaling argument together with Lemma~\ref{lem:unstablemesh}
and Remark~\ref{remark:magnitude_of_crit_wavenumber} allows to construct
a global singular solution as follows:
On each of the $N^2$ sub-quadrilaterals one chooses the
non-trivial solution constructed in the proof of Lemma \ref{lem:unstablemesh}.
It is easy to see that with the condition $kh = 2\hat{k}_{c}$ this global function is
then also a non-trivial solution to the global system of homogeneous equations.
\end{proof}
}

\subsection{Structured quadrilateral grids}

\label{subsection:quads} The present section is devoted to the study of
conforming Galerkin discretizations using quadrilateral elements in spatial
dimension two. We employ structured tensor-product meshes. The first result
concerns a $p$-version on \textit{one} quadrilateral element.

Throughout this section the following notation is employed:
For vectors $\mathbf{a}=\left(
a_{1},a_{2}\right)  $, $\mathbf{b}=\left(  b_{1},b_{2}\right)  $ in
$\mathbb{C}^{2}$ we use the notation $\mathbf{a}\cdot\mathbf{b}:=a_{1}%
b_{1}+a_{2}b_{2}$ without complex conjugation.
Furthermore, $\|\cdot\|_2$ denotes the Euclidean 2-norm.

\begin{theorem}
\label{thm:one_quad_p_version} Let $\Omega=\widehat{K}\subset\mathbb{R}^{2}$,
where $\widehat{K}$ denotes the reference quadrilateral with vertices
$(\pm1,\pm1)^{T}$. Consider the Galerkin discretization \eqref{GalDisc} of
\eqref{varformrobin} with polynomials of degree $p\geq1$ on $\widehat{K}$.
Then, for any $k\in\overset{\circ}{\mathbb{R}}$ the matrix $\mathbf{A}_{k}$ in
\eqref{Kk} is regular.
\end{theorem}
\begin{proof}
As in the proof of Theorem \ref{thm:1d_ucp} it suffices to show that any
solution for the homogeneous problem is already trivial. Throughout the proof,
we denote by $S$ the finite element space, i.e., the space of polynomials of
total degree $p$ on $\widehat{K}$. Let $u\in S$ be a solution to the
homogeneous equations. We again have $u=0$ on $\partial\widehat{K}$, see
equation \eqref{eq:impzero}. Therefore, $u\in S$ solves
\begin{equation}
(\nabla u,\nabla v)-k^{2}(u,v)=0\quad\forall v\in S.\label{eq:hom_quads}%
\end{equation}
{
The proof relies, similarly to Theorem \ref{thm:1d_ucp},
on the choice of a special kind of test functions, i.e. Morawetz-multipliers.
}
Note that for any $a$, $b$, $c$, $d\in\mathbb{R}$ we have $(ax+b)u_{x}\in S$
and $(cy+d)u_{y}\in S$. Therefore, with $%
\boldsymbol{\rho}
:=(ax+b,cy+d)^{T}$ the function $v=%
\boldsymbol{\rho} \cdot\nabla u\in S$ is a valid test function. Choosing $v=%
\boldsymbol{\rho} \cdot\nabla u$ in \eqref{eq:hom_quads}, passing to the real part, integrating
by parts together with the facts that
$2\Re(w\overline{w}_{x})=\partial_{x}(|w|^{2})$ and
$2\Re(w\overline{w}_{y})=\partial_{y}(|w|^{2})$
for sufficiently smooth functions $w$ and
employing the boundary conditions of $u$, we find
\begin{align*}
0 &  =\Re\,\left[  (\nabla u,\nabla(
\boldsymbol{\rho}\cdot\nabla u))-k^{2}(u,
\boldsymbol{\rho}\cdot\nabla u)\right]  \\
&  =\Re\,\left[  (\nabla u,\nabla%
\boldsymbol{\rho} \,\nabla u)+(\nabla u,\nabla^{2}u\,
\boldsymbol{\rho})-k^{2}(u, \boldsymbol{\rho}
\cdot\nabla u)\right]  \\
&  =\Re\,\left[  (\nabla u,\nabla\boldsymbol{\rho}
\,\nabla u)\right]  +\frac{1}{2}(\nabla\lVert\nabla u\rVert_2^{2},%
\boldsymbol{\rho})-\frac{k^{2}}{2}(\nabla|u|^{2},%
\boldsymbol{\rho})\\
&  =\Re\,\left[  (\nabla u,\nabla%
\boldsymbol{\rho}
\,\nabla u)\right]  -\frac{1}{2}(\lVert\nabla u\rVert_2^{2},\nabla\cdot%
\boldsymbol{\rho}
)+\frac{1}{2}\langle\lVert\nabla u\rVert_2^{2},%
\boldsymbol{\rho}
\cdot\mathbf{n}\rangle+\frac{k^{2}}{2}(|u|^{2},\nabla\cdot%
\boldsymbol{\rho}
)-\underbrace{\frac{k^{2}}{2}\langle|u|^{2},
\boldsymbol{\rho}
\cdot\mathbf{n}\rangle}_{=0}.
\end{align*}
The choice $
\boldsymbol{\rho}
=(1-x,1-y)^{T}$, results in $\nabla%
\boldsymbol{\rho}
-1/2\,\nabla\cdot%
\boldsymbol{\rho}
\,I$ being the zero matrix. Furthermore, $\nabla\cdot
\boldsymbol{\rho}
=-2$. We therefore have
\[
\frac{1}{2}\langle\lVert\nabla u\rVert_2^{2},%
\boldsymbol{\rho}
\cdot\mathbf{n}\rangle-k^{2}\lVert u\rVert_{L^{2}(\widehat{K})}^{2}=0.
\]
We find $u\equiv0$ once we have shown $\boldsymbol{\rho}
\cdot\mathbf{n}\leq0$. Since $%
\boldsymbol{\rho}
\cdot\mathbf{n}=0$ on the top-right part of $\partial\widehat{K}$ and $%
\boldsymbol{\rho}
\cdot\mathbf{n}=-2$ on the bottom-left part, we can conclude $u\equiv0$.
\end{proof}

The natural next step is to consider an axial parallel quadrilateral domain
$\Omega\subset\mathbb{R}^2$ and use a mesh, which consists of \textit{one}
corridor of elements, see Figure \ref{figure:one_corridor_quads}.
\begin{figure}[ptb]
	\centering
	\input{supplements/one_corridor_quads.tex}
	\caption{Example of a mesh as in Theorem \ref{thm:one_corridor_quads}.}
	\label{figure:one_corridor_quads}
\end{figure}
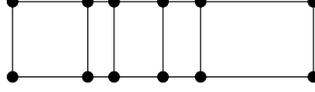
\begin{theorem}
\label{thm:one_corridor_quads} Let $\Omega\subset\mathbb{R}^{2}$ be an axial
parallel quadrilateral. Consider the Galerkin discretization \eqref{GalDisc}
of \eqref{varformrobin} with conforming finite elements with a mesh consisting
of a corridor of axial parallel quadrilaterals as depicted in Figure
\ref{figure:one_corridor_quads} and polynomial degree $p \geq1$. Then, for any
$k \in\overset{\circ}{\mathbb{R}}$ the matrix $\mathbf{A}_{k}$ in \eqref{Kk}
is regular.
\end{theorem}
\begin{proof}
Without loss of generality, we may assume that two of the sides of $\Omega$
are on
the lines $\{y=1\}$ and $\{y=-1\}$. Choosing $v=u$ and passing to the
imaginary part again leads to $u=0$ on $\Gamma$. We also find
\begin{equation}
0=(\lVert\nabla u\rVert_2^{2},1)-k^{2}(|u|^{2}%
,1).\label{eq:one_corridor_u_as_test_function}%
\end{equation}
Due to the axial symmetric mesh, the function $(1-y)u_{y}$ is a valid test
function. We can also write $(1-y)u_{y}=
\boldsymbol{\rho}
\cdot\nabla u$ with $%
\boldsymbol{\rho}
=(0,1-y)^{T}$. Proceeding as in the proof of Theorem
\ref{thm:one_quad_p_version} we find
\begin{align*}
0 &  =\Re\,\left[  (\nabla u,\nabla(%
\boldsymbol{\rho}
\cdot\nabla u))-k^{2}(u,%
\boldsymbol{\rho}
\cdot\nabla u)\right]  \\
&  =(\nabla u,\nabla%
\boldsymbol{\rho}
\,\nabla u)-\frac{1}{2}(\lVert\nabla u\rVert_2^{2},\nabla\cdot%
\boldsymbol{\rho}
)+\frac{1}{2}\langle\lVert\nabla u\rVert_2^{2},%
\boldsymbol{\rho}
\cdot\mathbf{n}\rangle+\frac{k^{2}}{2}(|u|^{2},\nabla\cdot%
\boldsymbol{\rho}
)-\frac{k^{2}}{2}\langle|u|^{2},%
\boldsymbol{\rho}
\cdot\mathbf{n}\rangle.
\end{align*}
Again we find $\langle|u|^{2},%
\boldsymbol{\rho}
\cdot\mathbf{n}\rangle=0$. Furthermore, note that
\[
\nabla%
\boldsymbol{\rho}
=%
\begin{pmatrix}
0 & 0\\
0 & -1
\end{pmatrix}
\qquad\text{ and }\qquad\nabla\cdot%
\boldsymbol{\rho}
=-1.
\]
Finally $
\boldsymbol{\rho}
\cdot\mathbf{n}=-2$ on the line $\{y=-1\}$ and vanishes on all other sides of
the boundary. Therefore, the above further simplifies to
\begin{equation}
0=\frac{1}{2}(|u_{x}|^{2},1)-\frac{1}{2}(|u_{y}|^{2},1)-%
\langle|u_{y}|^{2},1\rangle_{\{y=-1\}}-\frac{k^{2}}{2}(|u|^{2}%
,1).\label{eq:one_corridor_nabla_u_as_test_function}%
\end{equation}
Multiplying \eqref{eq:one_corridor_u_as_test_function} by $-1/2$ and adding
\eqref{eq:one_corridor_nabla_u_as_test_function} leads to
\[
0=-(|u_{y}|^{2},1)-\langle|u_{y}|^{2},1\rangle_{\{y=-1\}}.
\]
Consequently, $u_{y}\equiv0$. Combined with the fact that $u$ vanishes on the
boundary we find $u\equiv0$, which concludes the proof.
\end{proof}
\begin{figure}[ptb]
	\centering
	\input{supplements/two_corridor_quads.tex}
	\caption{Example of a mesh as in Theorem \ref{thm:two_corridor_quads}.}
	\label{figure:two_corridor_quads}
\end{figure}
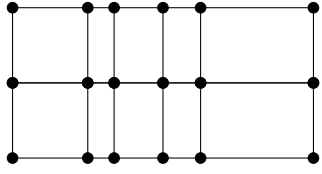

\begin{theorem}
\label{thm:two_corridor_quads} Let $\Omega\subset\mathbb{R}^{2}$ be an axial
parallel quadrilateral. Consider the Galerkin discretization \eqref{GalDisc}
of \eqref{varformrobin} with conforming finite elements with a mesh consisting
of two corridors of axial parallel quadrilaterals as depicted in Figure
\ref{figure:two_corridor_quads}. Then, for any $k \in\overset{\circ
}{\mathbb{R}}$ the matrix $\mathbf{A}_{k}$ in \eqref{Kk} is regular.
\end{theorem}
\begin{proof}
Without loss of generality, we may assume that the dividing line between the
two
corridors to be located on the line $\{y=0\}$. Let the two sides parallel to
the \(x\)-axis lie on the lines $\{y=a\}$ and $\{y=-b\}$, for some $a$, $b>0$.
Let $U\subset\Omega$ and $D\subset\Omega$ denote the upper and lower corridor
respectively. Again the proof relies on an appropriate choice of test
functions. These are the global function $v=u$, as well as $v=-2yu_{y}$
localized on $U$ and $D$ respectively, i.e., extended by zero. This
localization is again a valid test function since $v=-2yu_{y}$ is piecewise
polynomial and conforming since $v=0$ on the line $\{y=0\}$. Analogous
integration by parts as in the proof of Theorem \ref{thm:two_corridor_quads}
we find the following three equations to hold:
\begin{align*}
0  &  =(|u_{x}|^{2},1)_{\Omega}+(|u_{y}|^{2},1)_{\Omega}-k^{2}(|u|^{2}%
,1)_{\Omega},\\
0  &  =(|u_{x}|^{2},1)_{U}-(|u_{y}|^{2},1)_{U}-\langle|u_{y}|^{2}%
,1\rangle_{\{y=a\}}-k^{2}(|u|^{2},1)_{U},\\
0  &  =(|u_{x}|^{2},1)_{D}-(|u_{y}|^{2},1)_{D}-\langle|u_{y}|^{2}%
,1\rangle_{\{y=-b\}}-k^{2}(|u|^{2},1)_{D}.
\end{align*}
Adding the equations for $U$ and $D$ and subtracting the one on the whole
domain $\Omega$ again gives $u_{y}\equiv0$, which concludes the proof, with
the same arguments as in the proof of Theorem \ref{thm:one_corridor_quads}.
\end{proof}

The previous results relied on the use of appropriate \textit{global} test
functions. The remainder of this section is concerned with discretizations
employing quadrilaterals on structured Cartesian meshes. To that end let
$\widehat{K}$ again denote the reference quadrilateral with vertices
$(\pm1,\pm1)^{T}$. The setup is such that the bottom-left part of the boundary
of $\widehat{K}$ is part of the boundary $\Gamma$ of the computational domain
$\Omega$, again itself an axial parallel quadrilateral. The upper right part
of the boundary of $\widehat{K}$ is therefore inside the domain $\Omega$. Our
argument will be a localized one, i.e., we consider only test functions $v$
whose support is given by $\widehat{K}$. To that end let $\mathcal{Q}%
_{\mathrm{BL}}^{p}(\widehat{K})$ denote that space of polynomials of total
degree $p$, which are zero on the bottom-left part of the boundary of
$\widehat{K}$. Analogously let $\mathcal{Q}_{\mathrm{TR}}^{p}(\widehat{K})$
denote that space of polynomials of total degree $p$, which are zero on the
top-right part of the boundary of $\widehat{K}$. These spaces are therefore
given by
\begin{align*}
\mathcal{Q}_{\mathrm{BL}}^{p}(\widehat{K}) &  =\mathrm{span}\left\{
(1+x)(1+y)x^{i}y^{j}\colon0\leq i,j\leq p-1\right\},  \\
\mathcal{Q}_{\mathrm{TR}}^{p}(\widehat{K}) &  =\mathrm{span}\left\{
(1-x)(1-y)x^{i}y^{j}\colon0\leq i,j\leq p-1\right\}  .
\end{align*}
To perform a localized argument, we only test with functions $v\in
\mathcal{Q}_{\mathrm{TR}}^{p}(\widehat{K})$. Note that the Galerkin solution
$u$ vanishes on the bottom-left part of the boundary of $\widehat{K}$ and
therefore $\left.  u\right\vert _{\widehat{K}}\in\mathcal{Q}_{\mathrm{BL}}%
^{p}(\widehat{K})$. These considerations now lead to the question if a
solution $u\in\mathcal{Q}_{\mathrm{BL}}^{p}(\widehat{K})$ of
\begin{equation}
(\nabla u,\nabla v)_{\widehat{K}}-k^{2}(u,v)_{\widehat{K}}=0\qquad\forall
v\in\mathcal{Q}_{\mathrm{TR}}^{p}(\widehat{K}%
)\label{eq:non_reduced_local_equation_quads}%
\end{equation}
can only be $u\equiv0$. Upon introducing $b_{\mathrm{i}}=(1+x)(1+y)$ and
$b_{\mathrm{o}}(x,y)=(1-x)(1-y)$, we can write $u=b_{\mathrm{i}}u_{r}$ and
$v=b_{\mathrm{o}}v_{r}$ for polynomials $u_{r}$, $v_{r}\in\mathcal{Q}%
^{p-1}(\widehat{K}):=\mathrm{span}\left\{  x^{i}y^{j}\colon0\leq i,j\leq
p-1\right\}  $. The above is therefore equivalent to whether a solution
$u_{r}\in\mathcal{Q}^{p-1}(\widehat{K})$ to
\begin{equation}
(\nabla(b_{\mathrm{i}}u_{r}),\nabla(b_{\mathrm{o}}v_{r}))_{\widehat{K}}%
-k^{2}(b_{\mathrm{i}}u_{r},b_{\mathrm{o}}v_{r})_{\widehat{K}}=0\qquad\forall
v_{r}\in\mathcal{Q}^{p-1}(\widehat{K})\label{eq:reduced_equation_quads}%
\end{equation}
can only be $u_{r}\equiv0$. The answer to this question is yes for $p=1$, $2$,
$3$. However, for $p=4$ and some $k > 0$ such exists a non-trivial solution.
The consequences
of this are twofold: Firstly, on a structured Cartesian grid, the Galerkin
discretizations for $p=1$, $2$, $3$ are well posed for any $k\in
\mathbb{R}\backslash\left\{  0\right\}  $. Secondly, for $p \geq 4$ a \textit{localized}
argument based on the appropriate choice of test functions as in the
one-dimensional case, see Theorem \ref{thm:1d_ucp}, is not possible in two
dimensions, for all wave numbers simultaneously.

\begin{lemma}
\label{lemma:quads_lower_order} For $p=1$, $2$, $3$ the only solution to
\eqref{eq:reduced_equation_quads} is the trivial solution $u_{r}\equiv0$, for
every $k\in\overset{\circ}{\mathbb{R}} $. For $p=4$ and some
$k>0$ there exists a non-trivial solution to
\eqref{eq:reduced_equation_quads}.
\end{lemma}
\begin{proof}
The proof is an algebraic one. We calculate the matrix corresponding to the
system of linear equations \eqref{eq:reduced_equation_quads} explicitly. To
that
end we choose the monomial
basis of $\mathcal{Q}^{p-1}(\widehat{K})$, i.e., for $p=1$ the basis is given
by $\{1\}$, for $p=2$ by $\{1,y,x,xy\}$ and for $p=3$ by $\{1,y,y^{2}%
,x,xy,xy^{2},x^{2},x^{2}y,x^{2}y^{2}\}$. For $p=1$, i.e., a $1\times1$ system,
we find for the determinant the polynomial
\[
-\frac{16}{3}-k^{2}\frac{16}{9},
\]
which is strictly negative for any $k\in\mathbb{R}$. For $p=2$ we find
\[
\mathrm{det}\left(
\begin{pmatrix}
-\frac{16}{3} & \frac{8}{3} & \frac{8}{3} & 0\\[6pt]%
-\frac{8}{3} & -\frac{64}{45} & 0 & \frac{8}{15}\\[6pt]%
-\frac{8}{3} & 0 & -\frac{64}{45} & \frac{8}{15}\\[6pt]%
0 & -\frac{8}{15} & -\frac{8}{15} & -\frac{16}{45}\\
&  &  &
\end{pmatrix}
-k^{2}%
\begin{pmatrix}
\frac{16}{9} & 0 & 0 & 0\\[6pt]%
0 & \frac{16}{45} & 0 & 0\\[6pt]%
0 & 0 & \frac{16}{45} & 0\\[6pt]%
0 & 0 & 0 & \frac{16}{225}\\
&  &  &
\end{pmatrix}
\right)  =\frac{65536\left(  k^{2}+4\right)  ^{2}\left(  k^{4}+8k^{2}%
+60\right)  }{4100625},
\]
which is again non-zero. For $p=3$ we find the determinant is given, up to a
positive factor, by
\[
-\left(  k^{6}+15k^{4}+270k^{2}+3150\right)  \left(  4k^{6}+60k^{4}%
+495k^{2}+450\right)  ^{2}.
\]
For $p=4$ we find the following polynomial to be a factor of the determinant
\[
(-3492720-161028k^{2}+41013k^{4}+10800k^{6}+810k^{8}+36k^{10}+k^{12})^{2},
\]
which has a positive real root.
\end{proof}

An immediate consequence of Lemma \ref{lemma:quads_lower_order} is the
following Theorem.

\begin{theorem}
\label{thm:quads_lower_order_tensor_product} Let $\Omega\subset\mathbb{R}^{2}$
be an axial parallel quadrilateral. Consider the Galerkin discretization
\eqref{GalDisc} of \eqref{varformrobin} with conforming finite elements with a
tensor-product mesh and polynomial degree $p=1$, $2$, $3$. Then, for any
$k\in\overset{\circ}{\mathbb{R}} $ the matrix
$\mathbf{A}_{k}$
in
\eqref{Kk} is regular.
\end{theorem}

\begin{proof}
Propagating through the mesh by applying Lemma \ref{lemma:quads_lower_order}
yields the result.
\end{proof}

\section{A discrete unique continuation principle for the Helmholtz
equation}\label{sec:ducp} In this section, we will introduce the Algorithm MOTZ
(\textquotedblleft marching-of-the-zeros\textquotedblright), which mimics a
discrete
unique
continuation principle for the Helmholtz equation\footnote{A related notion
for a discrete unique continuation principle is used in the context of the
lattice Schr\"{o}dinger equation, see e.g., \cite{li2021andersonbernoulli}%
.} for \(d=2\) and a triangular mesh \(\mathcal{K}_{\mathcal{T}}\). We restrict
ourselves to the case where $p=1$ and use the notation as in
Remark \ref{rmk:notation}. The
discrete problem then reads
\begin{equation}
\text{find }u_{S}\in S\text{ such that }a_{k}\left(  u_{S},v\right)  =F\left(
v\right)  ,\quad\forall v\in S\label{GalDisc2},
\end{equation}
where \(S=S^1_{\mathcal{T}}\) is the space of continuous and piecewise linear
functions with
respect to a conforming triangulation \(\mathcal{K}_{\mathcal{T}}\) on
\(\Omega\).

\begin{notation}
In the following, we skip the polynomial degree $p$ in the notation and write
short $b_{z}$ for $b_{z,p}=b_{z,1}$.
\end{notation}

\begin{definition}
\label{def:transmissiondegree} Let $\mathcal{N}_{1}\subset\mathcal{N}$ be a
subset of nodes. For a node $z^{\prime}\in\mathcal{N}_{1}$ we define the
\emph{transmission degree} with respect to $\mathcal{N}$ by
\[
\mathrm{deg}(z^{\prime},\mathcal{N}_{1})=\operatorname*{card}\left\{
z\in\mathcal{N}\setminus\mathcal{N}_{1}\mid\lbrack z,z^{\prime}]\in
\mathcal{E}_{\Omega}\right\}  .
\]

\end{definition}


\begin{lemma}
\label{lem:lducp} Let $d=2$ and $\mathcal{N}_{\operatorname*{test}}%
\subset\mathcal{N}$ a subset of nodes. Let $z\in\mathcal{N}%
_{\operatorname*{dof}}:=\mathcal{N}\setminus\mathcal{N}_{\operatorname*{test}}$
have the property: there exists
$z^{\prime}\in\mathcal{N}_{\operatorname*{test}}$ with

\begin{enumerate}
\item[(a)] $E=[z,z^{\prime}]\in \mathcal{E}_{\Omega}$ and $\alpha_{E}\leq\pi$,
\label{item:lducp_angle_condition}
\item[(b)] $\deg(z^{\prime},\mathcal{N}_{\operatorname*{test}})=1$,
\label{item:lducp_degree_condition}
\end{enumerate}
(see Figure \ref{fig:anglecond}). Let $u\in S$ be a solution of
\eqref{GalDisc2}. Then the following
implication holds%
\begin{equation}
u(z^{\prime\prime})=0~~\forall z^{\prime\prime}\in\mathcal{N}%
_{\operatorname*{test}}\implies u(z)=0.\label{eq:lducp}%
\end{equation}

\end{lemma}

\begin{figure}
	\begin{center}
		\begin{tikzpicture}[scale = 2]
		\coordinate (orig) at (0,0);
		\coordinate (top) at (0,1);
		\coordinate (left) at (-1,0.3);
		\coordinate (right) at (1.2,0.2);
		\draw (0,0) -- (1.2,0.2) -- (0,1);
		\draw (0,0) -- (-1,0.3)  -- (0,1) -- (0,0);

		\node[regular polygon, regular polygon sides=3, inner sep=1.5pt, draw = red, fill = red] at (0, 1) {};

		\draw[draw = orange, fill = orange]  (1.2,0.2) circle[radius = 2pt];
		\draw[draw = orange, fill = orange]  (-1,0.3) circle[radius = 2pt];
		\draw[draw = orange, fill = orange]  (0,0) circle[radius = 2pt];
		\pic [draw, "$\alpha_+$", angle eccentricity=1.5] {angle =
			orig--left--top};
		\pic [draw, "$\alpha_-$", angle eccentricity=1.5] {angle =
			top--right--orig};
		\draw (top) node[anchor = south west] {$z$};
		\draw (orig) node[anchor = north west] {$z^\prime$};
		\end{tikzpicture}
	\end{center}
	\caption{The setting for Lemma \ref{lem:lducp}\\ (\protect\tikz
		\protect\draw[draw=orange, fill=orange]  (0,0) circle[radius = 3pt];
		\(\in
		\Nout\)
		and \trianglecolored{red} \(\in
		\Nin\)
		).}
	\label{fig:anglecond}
\end{figure}

\begin{notation}
We generally denote by $\mathcal{N}_{\operatorname*{test}}\subset\mathcal{N}$
a subset of nodes $z$, where we already know that the solution $u$ of
\eqref{GalDisc2} is zero, i.e., $u_{z}=0$ as in the assumption in
\eqref{eq:lducp}, but where we have a basis function $b_{z}$ that can be used
as test function. An example is $\mathcal{N}_{\operatorname*{test}%
}=\mathcal{N}\cap\Gamma$. Indeed, from \eqref{eq:impzero}, we know that
$u_{z}=0\ $for all $z\in\mathcal{N}\cap\Gamma$.
\end{notation}

\begin{remark}
\label{rmk:angelcond} If the weakly acute angle condition
\begin{equation}
\alpha_{E}\leq\pi\label{anglecond}%
\end{equation}
is violated for some edge $E\in\mathcal{E}_{\Omega}$ one can take the midpoint
of $E$ as a new mesh point and bisect the adjacent triangles.
For general dimension $d$, an analogous angle criterion as (\ref{anglecond})
can be formulated (see \cite[Lem. 2.1]{XuZikatanov99}).
\end{remark}
\begin{proof}[Proof of Lemma \ref{lem:lducp}] Let $b_{z^{\prime}}$ be the Lagrange
basis function for the node $z^{\prime}$ and $b_{z}$ the one for $z$. Then
(\ref{eq:lducp}) is equivalent to showing
\begin{equation}
\left(  \nabla b_{z},\nabla b_{z^{\prime}}\right)  -k^{2}\left(
b_{z},b_{z^{\prime}}\right)  \neq0. \label{condequiv}%
\end{equation}
By Lemma \ref{lem:anglecond} we have that $\left(  \nabla b_{z},\nabla
b_{z^{\prime}}\right)  \leq0$ if and only if $\alpha_{E}\leq\pi$. Since
$b_{z}$ and $b_{z^{\prime}}$ are positive in $\operatorname*{int}\left(
\tau_{+}\right)  \cup\operatorname{int}\left(  \tau_{-}\right)  $, we conclude
that (\ref{condequiv}) holds.
\end{proof}

\subsection{A first checking algorithm and numerical experiments}

In this section we present a main result of the paper
(cf. Theorem
\ref{TheoDUCPcheck}): If the algorithm MOTZ
(Algorithm \ref{alg:check_ducp_local}) returns
\texttt{certified} then we conclude that the discrete problem is well posed. On
the other hand,
if the output is \texttt{critical} then this means that the discretization
might lead to a
singular matrix. In the latter case a slight modification of the mesh
(using bisection or flip of an edge) may be applied to receive a regular system
matrix. We introduce the
following notation

\begin{itemize}
\item $\mathcal{N}_{\operatorname*{test}}$: The set of $z\in\mathcal{N}$ where
$u(z)=0$.
\item $\mathcal{N}_{\operatorname*{dof}}$: The complement of $\mathcal{N}%
_{\operatorname*{test}}$.
\end{itemize}

\begin{figure}
	\centering
	\input{supplements/tikz_algo1}
	\caption{Procedure described in Alg. \ref{alg:check_ducp_local} (\protect\tikz
		\protect\draw[draw=orange, fill=orange]  (0,0) circle[radius = 3pt];
		\(\in
		\Nout\)
		and \trianglecolored{red} \(\in
		\Nin\)
		). }
\end{figure}

The idea of Algorithm \ref{alg:check_ducp_local} is the following: We
initialize the algorithm with $\mathcal{N}_{\operatorname*{test}}%
=\mathcal{N}\cap\Gamma$ and $\mathcal{N}_{\operatorname*{dof}}=\mathcal{N}%
\setminus\mathcal{N}_{\operatorname*{test}}$. Recall that for $z^{\prime}%
\in\mathcal{N}_{\operatorname*{test}}$ with $\deg\left(  z^{\prime
},\mathcal{N}_{\operatorname*{test}}\right)  =1$ (cf. Definition
\ref{def:transmissiondegree}) there exists one and only one
$z^{\prime}\in\mathcal{N}_{\operatorname*{dof}}$ such that the edge $E=\left[
z,z^{\prime}\right]  $ belongs to $\mathcal{E}_{\Omega}$. Lemma
\ref{lem:lducp}, together with \eqref{eq:impzero} then implies that $u(z)=0$.
We update the sets $\mathcal{N}_{\operatorname*{test}}=\mathcal{N}%
_{\operatorname*{test}}\cup\{z\}$ and $\mathcal{N}_{\operatorname*{dof}%
}=\mathcal{N}_{\operatorname*{dof}}\setminus\{z\}$ accordingly and repeat the
same procedure. After each step one has
\[
u(z^{\prime\prime})=0~~\forall z^{\prime\prime}\in\mathcal{N}%
_{\operatorname*{test}}.
\]
If $\mathcal{N}_{\operatorname*{test}}=\mathcal{N}$ the algorithm stops.

We remark that we do not stop the algorithm if
the weakly acute angle condition \eqref{anglecond} is not satisfied.
Instead, we assign to the corresponding edge the property \texttt{acute(E) =
false} (line 7 of Alg. \ref{alg:check_ducp_local}). If angle conditions are not
satisfied
everywhere, but the
algorithm ends with \(\Nout = \N\), we bisect the relevant, adjacent
triangles in a post-processing step as explained in Remark \ref{rmk:angelcond}
using Algorithm \ref{alg:anglebisection} in the
end. This is possible since two transmission edges never share an adjacent
triangle.
\begin{algorithm}
	\caption{\alg}
	\label{alg:check_ducp_local}
	\noindent
	\hspace*{\algorithmicindent} \textbf{Input:} \(\Nout\), \(\N\),
	\(\mathcal{E}_{\Omega}\) \\
	\hspace*{\algorithmicindent} \textbf{Output:} MOTZ\_result,
	MOTZ\_trans, MOTZ\_angle, \(\Nout\), \(\Nin\),
	\(\mathcal{E}_{\Omega}\)
	\begin{algorithmic}[1]
		\STATE MOTZ\_result = \textbf{certified}, MOTZ\_trans = \(\TRUE\),
		MOTZ\_angle = \(\TRUE\)
		\STATE \(\Nin:= \N\setminus \Nout\)
		\WHILE{\(\Nin \neq \emptyset\)}
		\IF{$\exists (z,z^\prime)\in\Nin\times\Nout$: \(E=[z,z^\prime]\in
		\mathcal{E}_{\Omega}\) \(\wedge\) \(\deg(z^\prime,
		\Nout) = 1\)}
		\STATE trans\((E)=\TRUE\)
		\IF{\(\alpha_E > \pi\)}
		\STATE acute\((E)= \FALSE\)
		\STATE MOTZ\_angle = \(\FALSE\)
		\ENDIF
		\STATE \(\Nout = \Nout\cup \{z\}\)
		\STATE \(\Nin = \Nin \setminus \{z\}\)
		\ELSE
		\STATE MOTZ\_trans = \(\FALSE\) ; \textbf{STOP}
		\ENDIF
		\ENDWHILE
		\IF{MOTZ\_angle == \(\FALSE\) \(\vee\) MOTZ\_trans == \(\FALSE\)}
		\STATE MOTZ\_result = \textbf{critical}
		\ENDIF
	\end{algorithmic}

\end{algorithm}

\begin{algorithm}
	\caption{correct\_angle\_condition}
	\label{alg:anglebisection}
	\noindent
	\hspace*{\algorithmicindent} \textbf{Input:} \(\N\),
	\(\mathcal{E}_{\Omega}\) \\
	\hspace*{\algorithmicindent} \textbf{Output:} updated \(\N\),
	\(\mathcal{E}_{\Omega}\)
	\begin{algorithmic}[1]
		\FORALL{\(E\in \mathcal{E}_{\Omega}\) with \(\operatorname{trans}(E) =
		\TRUE\)
			\(\wedge\) \(\alpha_E >\pi\)}
		\STATE bisect \(E\) until weakly acute angle condition is satisfied
		\ENDFOR
	\end{algorithmic}
\end{algorithm}

\begin{lemma}
\label{lem:finitebisect} After a finite number of bisections the weakly acute
angle condition is satisfied.
\end{lemma}

\begin{proof}
We consider the setting depicted in Figure \ref{fig:finitebisect}. Let
$\tau=[A,B,C]$ be a triangle with corresponding angles $\alpha,\beta,\gamma$
such that $\alpha\leq\beta<\gamma$ and $\gamma>\frac{\pi}{2}$. After one
bisection the angle $\gamma$ splits in two angles $\gamma_{1}+\gamma
_{2}=\gamma$ (cf. Figure \ref{fig:finitebisect}) with the convention
$\gamma_{1}\leq\gamma_{2}$. Then $\gamma_{1}\geq\alpha$ \cite[Eq.
(4) and (6)]{Stamm1998}. This implies that $\gamma_{1}\leq\gamma_{2}=\gamma
-\gamma_{1}\leq\gamma-\alpha$. If $\gamma_{2}>\frac{\pi}{2}$, then another
bisection implies with the same arguments that $\gamma_{3},\gamma_{4}%
\leq\gamma-2\alpha$, since $\beta\geq\alpha$ as well as $\tilde{\alpha}%
=\alpha+\alpha_{1}\geq\alpha$. We conclude that after a finite number $n$ of
bisections (depending on the smallest angle in the triangle), all angles
$\gamma_{i}$, which subdivide $\gamma$, i.e., $\sum_{i=1}^{n+1}\gamma
_{i}=\gamma$ satisfy $\gamma_{i}\leq\frac{\pi}{2}$.
\end{proof}
\begin{figure}
	\centering
	\begin{minipage}{0.5\textwidth}
		\centering
		\begin{center}
			\begin{tikzpicture}[scale = 1.1]
			\coordinate (orig) at (0,0);
			\coordinate (top) at (1,2);
			\coordinate (left) at (-3,0);
			\coordinate (right) at (3,0);
			\coordinate (midright) at (1.5,0);
			\draw (orig) -- (right) -- (top);
			\draw (orig) -- (left)  -- (top);
			\draw (orig) -- (left)  -- (top) ;
			\draw[dashed] (top) -- (orig);
			\draw[dashed] (top) -- (midright);
			\pic [draw, red, "$\gamma_1$"{red}, angle eccentricity=1.5, angle
			radius =
			0.6cm] {angle =
				left--top--orig};
			\pic [draw, green, "$\gamma_2$"{green, shift={(0.5cm, 0.6cm)}},
			angle
			eccentricity=1.5, angle radius =
			0.7cm] {angle =
				orig--top--right};
			\pic [draw, blue, "$\gamma_3$"{blue}, angle
			eccentricity=1.5, angle radius =
			0.8cm] {angle = orig--top--midright};
			\pic [draw, purple, "$\gamma_4$"{purple}, angle
			eccentricity=1.5, angle radius =
			0.9cm] {angle = midright--top--right};
			\pic [draw, "$\gamma$"{shift={(0.1cm, 0.4cm)}},
			angle eccentricity=1.5] {angle =
				left--top--right};
			\pic [draw, "$\beta$", angle eccentricity=1.5] {angle =
				top--right--orig};
			\pic [draw, "$\alpha$", angle eccentricity=1.5, angle radius =
			0.7cm] {angle =
				orig--left--top};
			\pic [draw, "$\tilde{\alpha}=
			\alpha+\alpha_1$"{shift={(0.3,-0.3)}},
			angle
			eccentricity=1.5, angle radius =
			0.7cm] {angle =
				midright--orig--top};
			\draw (top) node[anchor = south] {$C$};
			\draw (left) node[anchor = east] {$A$};
			\draw (right) node[anchor = west] {$B$};
			\end{tikzpicture}
			\caption{Setting for Lemma \ref{lem:finitebisect}. }
			\label{fig:finitebisect}
		\end{center}
	\end{minipage}
	\hfill
	\begin{minipage}{0.4\textwidth}
		\centering
		\begin{tikzpicture}[scale = 2, rotate=-54]
		\coordinate (orig) at (0,0);
		\coordinate (top) at (0.2,2);
		\coordinate (left) at (-1,0.5);
		\coordinate (right) at (0.8,1.1);
		\coordinate (mid)  at (0.1,1);
		\draw (orig) -- (right) -- (top);
		\draw (orig) -- (left)  -- (top) -- (orig);
		\draw (orig) -- (left)  -- (top) -- (orig);
		\draw[dashed] (left) -- (mid)  -- (right);

		\fill[draw = red, fill=red] plot[mark=triangle*, mark options={rotate=54}] (top);
		\fill[draw = red, fill=red] plot[mark=triangle*, mark options={rotate=54}] (mid);

		\draw[draw = orange, fill = orange]  (right) circle[radius = 2pt];
		\draw[draw = orange, fill = orange]  (left) circle[radius = 2pt];
		\draw[draw = orange, fill = orange]  (orig) circle[radius = 2pt];
		\draw (top) node[anchor = south west] {$z$};
		\draw (orig) node[anchor = east] {$z^\prime$};
		\draw (mid) node[anchor = south] {$\tilde{z}$};
		\end{tikzpicture}
		\caption{The setting for Lemma \ref{lem:bisect1}. \\(\protect\tikz
			\protect\draw[draw=orange, fill=orange]  (0,0) circle[radius =
			3pt];\(\in \Nout\), \trianglecolored{red} \(\in \Nin\)).
		}
		\label{fig:bisect2}
	\end{minipage}
\end{figure}

\begin{lemma}
\label{lem:bisect1} Assume that the Algorithm MOTZ returns
\textnormal{\texttt{MOTZ\_angle = false}} and \textnormal{\texttt{MOTZ\_trans =
true}} for a given Galerkin
discretization as in
\eqref{GalDisc2}.
Running the bisection Algorithm \ref{alg:anglebisection}, does not change the
outcome of \texttt{MOTZ\_trans}.
\end{lemma}

\begin{proof}
Let $E=[z^{\prime},z]$ be a transmission edge such that $\alpha_{E}>\pi$.
After one bisection of $E$, a new vertex $\tilde{z}$ on $E$ is added (cf. Figure
\ref{fig:bisect2}). In order to have the same outcome of Algorithm MOTZ\_trans
we
need to show that $u(z^{\prime})=u(z)=0$. Indeed, this follows from applying
Lemma \ref{lem:lducp} along the edge $[z^{\prime},\tilde{z}]$ and then
$[\tilde{z},z]$, provided the angle condition is met. If the angle
condition does not hold, the argument can be repeated for every subsequent
bisection. The algorithm stops after a finite number of bisections,
according to Lemma \ref{lem:finitebisect}.
\end{proof}

\begin{theorem}
\label{TheoDUCPcheck}Consider the Galerkin discretization of
(\ref{varformrobin}) by (\ref{GalDisc2}).
\begin{enumerate}[(a)]
	\item If MOTZ$(\mathcal{N}\cap\Gamma,\mathcal{N}%
,\mathcal{E}_{\Omega})$ returns the value \textnormal{\texttt{MOTZ\_result =
certified}}, then the system matrix $\mathbf{A}_{k}$
in (\ref{Kk}) is regular for \emph{any} \(k\in\overset{\circ}{\mathbb{R}}\).
\item If MOTZ$(\mathcal{N}\cap\Gamma,\mathcal{N}%
,\mathcal{E}_{\Omega})$ returns the values \textnormal{\texttt{MOTZ\_trans
= true}} and \textnormal{\texttt{MOTZ\_angle = false}}, then after running
Algorithm \ref{alg:anglebisection}, the system matrix $\mathbf{A}_{k}$ in
(\ref{Kk}) for the \emph{new} mesh is regular \emph{any}
\(k\in\overset{\circ}{\mathbb{R}}\).
\end{enumerate}
\end{theorem}
\begin{proof}
From $u\in S$ and (\ref{eq:impzero}) we know that $\left.  u\right\vert
_{\Gamma}=0$. Choose $z_{1}$, $z_{1}^{\prime}$ as defined in Algorithm
\ref{alg:check_ducp_local} MOTZ
and set $E=\left[  z_{1},z_{1}^{\prime}\right]  $. With Lemma \ref{lem:lducp},
we conclude that $u\left(  z_1^\prime\right)  =0$. By an inductive application
of this
argument to pairs ($z_{j}$, $z_{j}^{\prime}$) we obtain that $u$ is zero at
all $z\in\mathcal{N}$. This means that $u$ vanishes in all mesh points
$\mathcal{N}$. If \texttt{MOTZ\_trans = true}, but \texttt{MOTZ\_angle =
		false}, running Algorithm 2 makes sure that for the new mesh all
		relevant weakly acute angle conditions are met. This procedure does not
		change the output of \texttt{MOTZ\_trans} as proved in Lemma
		\ref{lem:bisect1}. Therefore, if we run Algorithm 1 with the new mesh
		the output will be \texttt{MOTZ\_result = certified} and the first
		statement of the theorem applies.
\end{proof}

\begin{remark}
	In general, a lower bound on the smallest angle of the mesh determines how
	many bisections are at most needed. In all examples that are presented in
	this publication, all weakly acute angle conditions were satisfied. In
	particular
		this means that the outcome \texttt{MOTZ\_result} solely depended on
		the
		outcome \texttt{MOTZ\_trans}, i.e. on the connectivity of the mesh. We
		also note that the algorithm could easily be modified in order to avoid
		edges which do not satisfy the weakly acute angle conditions (if
		possible). However, since acute angle conditions in our examples were
		always met, this was not implemented in our code.
\end{remark}

Figure \ref{fig:ex_corner_mesh} shows the finite element mesh of a non-convex
geometry with re-entrant corners. The boundary with Robin boundary conditions
is illustrated in blue. In order to determine if the Galerkin finite element
method \eqref{GalDisc2} for this mesh is well posed, we apply Algorithm
\ref{alg:check_ducp_local} with $\mathcal{N}_{\operatorname*{test}}$ being the nodes on the
boundary (cf. Figure \ref{fig:ex_corner_0}). The red nodes belong to
$\mathcal{N}_{\operatorname*{dof}}$.\newline In the subsequent figures the
evolution of the algorithm is shown. It successively tries to find nodes in
$\mathcal{N}_{\operatorname*{dof}}$ that have a neighbouring node in
$\mathcal{N}_{\operatorname*{test}}$ with transmission degree 1. Once such a
node has been identified, it is removed from $\mathcal{N}_{\operatorname*{dof}%
}$ and added to $\mathcal{N}_{\operatorname*{test}}$. In this example the
procedure can be repeated until $\mathcal{N}_{\operatorname*{dof}}$ is the
empty set and $\mathcal{N}_{\operatorname*{test}}=\mathcal{N}$, i.e., MOTZ will
return \texttt{MOTZ\_trans = true}. Since also all angle conditions are
satisfied the algorithm will return \texttt{certified}. Furthermore, due to
the regularity of the mesh,
nodes that satisfy condition (a) of Lemma \ref{lem:lducp} are easily found in
each step, since they are typically located next to the node that has been
removed from $\mathcal{N}_{\operatorname*{dof}}$ in the previous step.
\newline Note that the order in which nodes are removed from $\mathcal{N}%
_{\operatorname*{dof}}$ depends on the enumeration of the nodes in the mesh.
The outcome \texttt{MOTZ\_trans} however, is independent of the
node enumeration.
\begin{figure}[htb]
	\noindent
	\begin{subfigure}{.32\textwidth}
		\centering
		\includegraphics[width=1.1\linewidth]{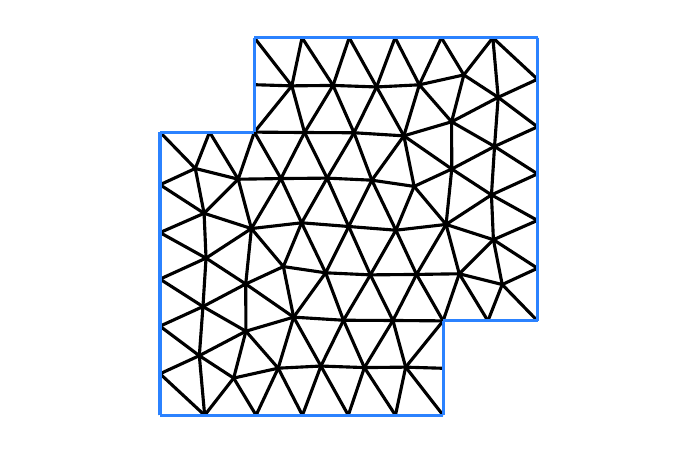}
		\caption{Mesh with boundary}
		\label{fig:ex_corner_mesh}
	\end{subfigure}%
	\begin{subfigure}{.32\textwidth}
		\centering
 		\includegraphics[width=1.1\linewidth]{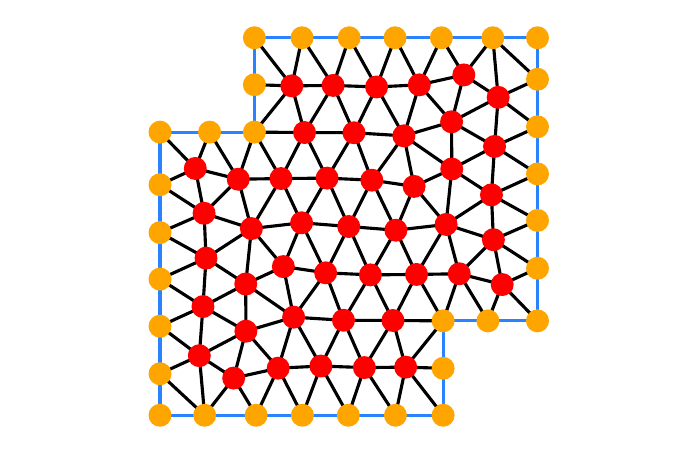}
		\caption{Step 0: initial state}
		\label{fig:ex_corner_0}
	\end{subfigure}
	\begin{subfigure}{.32\textwidth}
		\centering
		\includegraphics[width=1.1\linewidth]{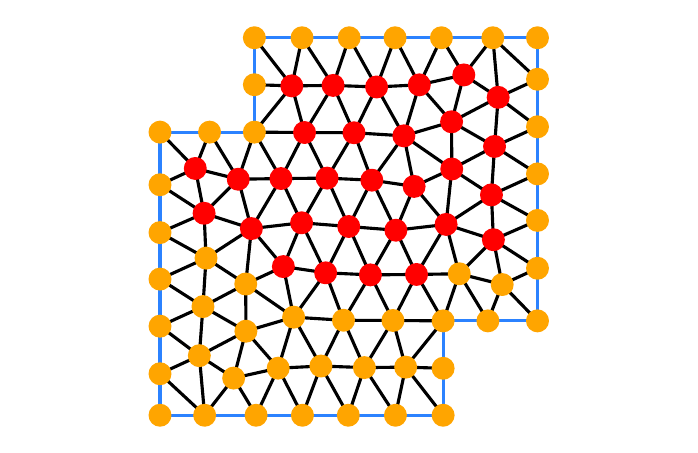}
		\caption{Step 15}
		\label{fig:ex_corner_15}
	\end{subfigure}
	\\[4mm]
	\begin{subfigure}{.32\textwidth}
		\centering
		\includegraphics[width=1.1\linewidth]{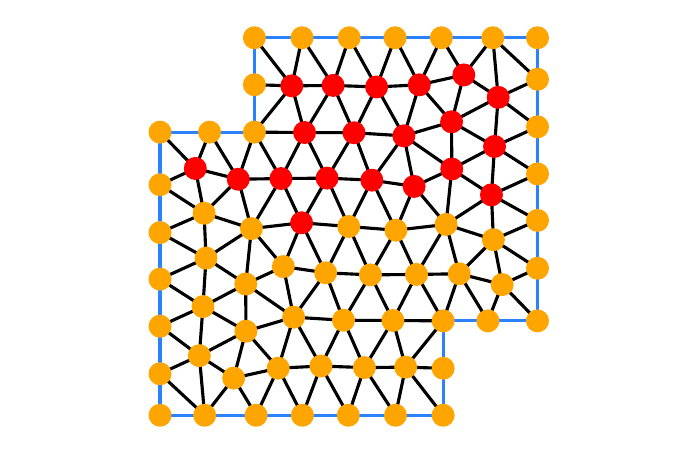}
		\caption{Step 25}
		\label{fig:ex_corner_25}
	\end{subfigure}%
	\begin{subfigure}{.32\textwidth}
		\centering
		\includegraphics[width=1.1\linewidth]{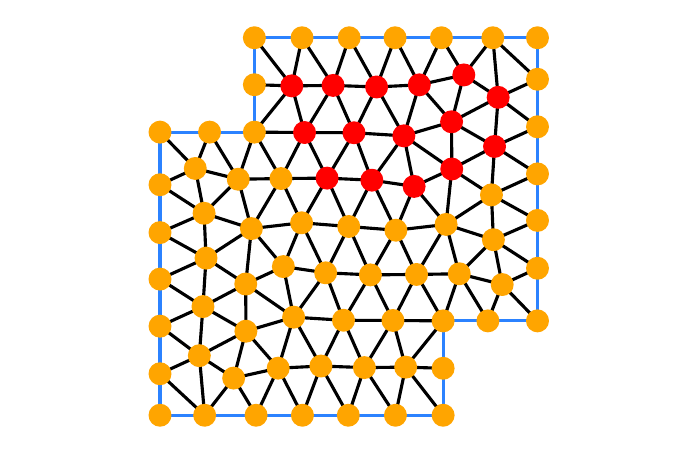}
		\caption{Step 30}
		\label{fig:ex_corner_30}
	\end{subfigure}
	\begin{subfigure}{.32\textwidth}
		\centering
		\includegraphics[width=1.1\linewidth]{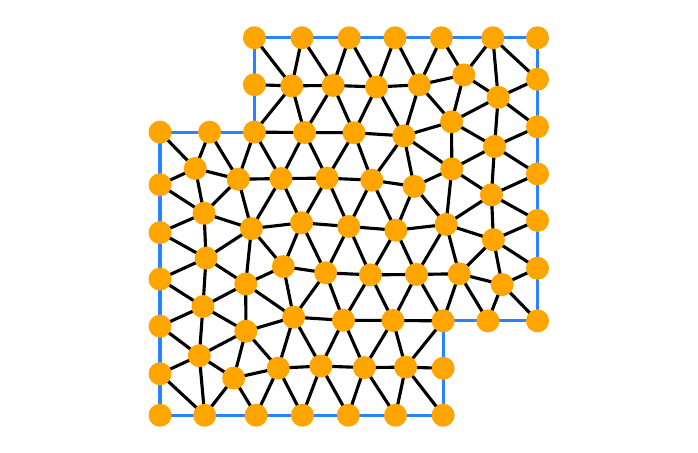}
		\caption{Step 45: final state}
		\label{fig:ex_corner_45}
	\end{subfigure}
	\caption{Example of a non-convex geometry and the evolution of Algorithm MOTZ (with result \texttt{MOTZ\_result = certified}).}
	\label{fig:mesh_corner}
\end{figure}

Figure \ref{fig:mesh_hole} shows the mesh of a geometry with one hole. As
before, the boundary with Robin boundary conditions is illustrated in blue (note, that the hole has Robin boundary conditions as well) and
we initialize the algorithm with $\mathcal{N}_{\operatorname*{test}%
}:=\mathcal{N}\cap\Gamma$. Also in this example, MOTZ returns
\texttt{MOTZ\_result = certified} which means that problem \eqref{GalDisc2} is
well posed. In
the following we refer to the nodes in $\mathcal{N}_{\operatorname*{dof}}$
(red nodes) that are connected by an edge to the boundary nodes of the inner
circle
as
`layer 1' nodes, `layer 0' being the boundary nodes on the circle.
Interestingly, none of the `layer 1' nodes can be marked orange initially,
since each boundary node on the circle is connected to at least two `layer 1'
nodes and therefore has a transmission degree larger than one. Thus, MOTZ has
to start from the outer boundary and successively moves towards the interior
boundary points. Only in step 62 of the algorithm (see Figure
\ref{fig:ex_hole_62}) one entry point into `layer 1' can be found. Note that
none of the other nodes in `layer 1' could have been marked orange at this
point, since all of the orange `layer 2' nodes, except the one, have
transmission
degree 2 (are connected to two red nodes).

\begin{figure}[htb]
	\noindent
	\begin{subfigure}{.5\textwidth}
		\centering
		\includegraphics[width=1.0\linewidth]{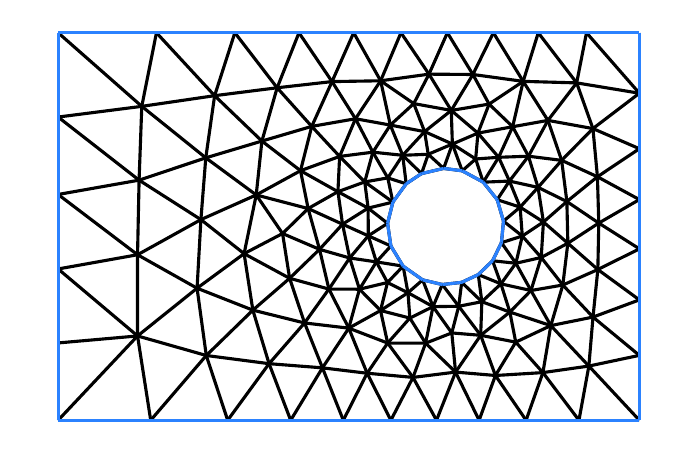}
		\caption{Mesh with boundary}
		\label{fig:ex_hole_mesh}
	\end{subfigure}%
	\begin{subfigure}{.5\textwidth}
		\centering
		\includegraphics[width=1.0\linewidth]{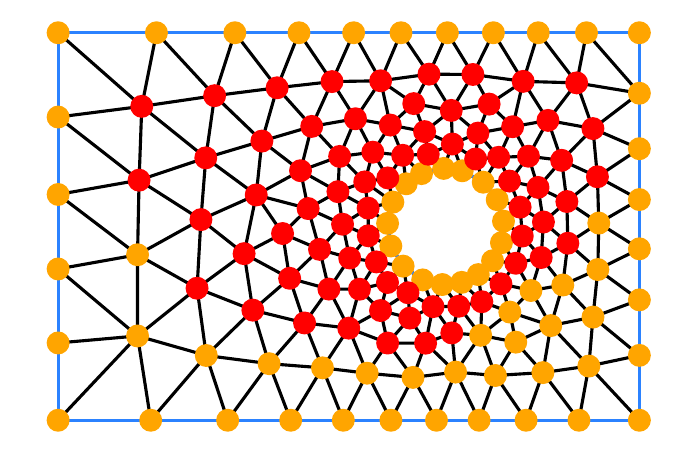}
		\caption{Step 20}
		\label{fig:ex_hole_20}
	\end{subfigure}
	\\[4mm]
	\begin{subfigure}{.5\textwidth}
		\centering
		\includegraphics[width=1.0\linewidth]{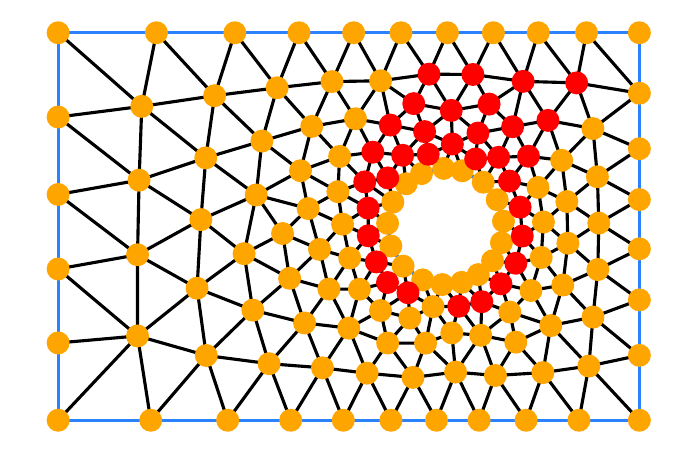}
		\caption{Step 62}
		\label{fig:ex_hole_62}
	\end{subfigure}%
	\begin{subfigure}{.5\textwidth}
		\centering
		\includegraphics[width=1.0\linewidth]{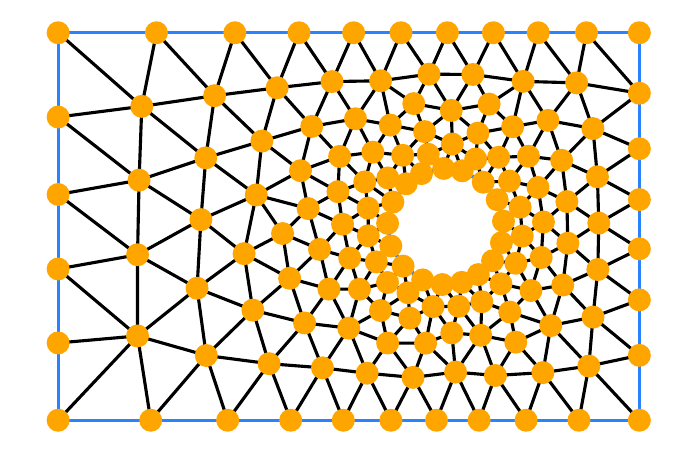}
		\caption{Step 95: final state}
		\label{fig:ex_hole_95}
	\end{subfigure}
	\caption{Example of a non-convex geometry with one hole and the evolution
	of Algorithm MOTZ (with result \texttt{MOTZ\_result = certified}).}
	\label{fig:mesh_hole}
\end{figure}

This suggests how a finite element mesh would need to look like in order for
MOTZ to give the result \texttt{MOTZ\_result = critical}. If the
mesh in Figure \ref{fig:ex_hole_mesh} was such that each `layer 1' node is
connected to exactly two nodes in `layer 0' and `layer 2', the algorithm would
not find any entry point into `layer 1' and would return \texttt{MOTZ\_result =
critical}.
However, in practice we could not produce such a mesh with standard mesh
generation tools due to the non-optimal quality of the desired mesh.\newline

We applied the Algorithm MOTZ to various geometries and mesh configurations
(sharp corners, complex geometries, strong local refinements). None of the
tested meshes that were produced by a standard mesh generation algorithm
(such as \cite{Schloemer2021,Schloemer2021a}),
actually led to the output \texttt{MOTZ\_result = critical} of the Algorithm
MOTZ.

\subsection{A mesh modification algorithm and numerical experiments}

\label{SecMeshEnrichment}

Even though the algorithm seems to return the output
\texttt{MOTZ\_result = certified} for most
shape regular meshes, one can construct examples where it returns a critical
result. In this section,
	we are in particular interested in the case where this is due to the
	output \texttt{MOTZ\_trans = false}, i.e. where no more edges which
	satisfy condition (a) of Lemma \ref{lem:lducp} could be found.
In these cases we need to modify the finite element mesh so that the
corresponding Galerkin discretization has a unique
solution. We propose the following
three simple mesh modification strategies, that can lead to a passing of the
checking algorithm:

\begin{itemize}
\item Re-building of the whole mesh with slightly modified mesh parameters.

\item Local refinement of the mesh across the interface of $\mathcal{N}%
_{\operatorname*{dof}}$ and $\mathcal{N}\backslash\mathcal{N}%
_{\operatorname*{dof}}$, where $\mathcal{N}_{\operatorname*{dof}}$ is the
result of MOTZ.

\item Application of Algorithm \ref{alg:mesh-refinement} (MOTZ\_flip), which
flips certain edges at the interface of $\mathcal{N}_{\operatorname*{dof}}$
and $\mathcal{N}\backslash\mathcal{N}_{\operatorname*{dof}}$.
\end{itemize}

If MOTZ returns \texttt{MOTZ\_trans = false} (together with the non-empty
set
$\mathcal{N}_{\operatorname*{dof}}$), the algorithm was not able to find any
more nodes that satisfy condition (a) of Lemma
\ref{lem:lducp}. The idea
behind all three strategies above is to alter existing or create new entry
points for the algorithm into the remaining set $\mathcal{N}%
_{\operatorname*{dof}}$. Re-building of the whole mesh with slightly different
parameters or a different meshing algorithm is a simple way of altering
triangles and edges which in turn might break up the constellations that lead
to the output \texttt{MOTZ\_result = critical}. If this is not possible or not
successful, a more targeted local refinement around a node on the boundary of $\mathcal{N}%
_{\operatorname*{dof}}$ might lead to new entry points and to a passing of the
checking algorithm.\newline A highly targeted approach to create new entry
points with minimal modifications to the original mesh is described in
Algorithm \ref{alg:mesh-refinement}. The idea is to detect those nodes on the
boundary of $\mathcal{N}_{\operatorname*{dof}}$ that have exactly two
connected nodes in $\mathcal{N}_{\operatorname*{test}}$, which in turn have
exactly one common node in $\mathcal{N}_{\operatorname*{test}}$ (see Figure
\ref{fig:angleflip}). Flipping the interior edge in this scenario, i.e.,
replacing $[z_{1},z_{2}]$ with $[z,\tilde{z}]$, will then increase the
transmission degree of $\tilde{z}$ by one. Since this will typically mean that
$\deg\left(  \tilde{z},\mathcal{N}_{\operatorname*{test}}\right)  =1$, this
mesh modification will create a new
entry point for MOTZ into $\mathcal{N}_{\operatorname*{dof}}$. Often, the
constellation described in Figure \ref{fig:angleflip} can be found multiple
times in a mesh. In Algorithm 3 we propose to compute a mesh quality score for
each potential edge flip (e.g. based on minimal angles of the resulting triangles). MOTZ is then rerun for the
modified mesh with the highest quality score.

\begin{figure}
	\begin{center}
		\begin{tikzpicture}[scale = 2]
		\coordinate (right) at (0,0);
		\coordinate (top) at (-1,0.5);
		\coordinate (bottom) at (-1,-0.5);
		\coordinate (left) at (-1.5,0);

		\coordinate (right2) at (0+3,0);
		\coordinate (top2) at (-1+3,0.5);
		\coordinate (bottom2) at (-1+3,-0.5);
		\coordinate (left2) at (-1.5+3,0);

		\draw (right) -- (top)  -- (bottom) -- (right);
		\draw (top)  -- (left) -- (bottom);

		\draw (right2) -- (top2) -- (left2)  -- (bottom2) -- (right2);
		\draw (right2)  -- (left2);

		\fill[draw = red, fill=red] plot[mark=triangle*] (right);
		\fill[draw = red, fill=red] plot[mark=triangle*] (right2);

		\draw[draw = orange, fill = orange] (left)  circle[radius = 2pt];
		\draw[draw = orange, fill = orange] (top)  circle[radius = 2pt];
		\draw[draw = orange, fill = orange] (bottom) circle[radius = 2pt];
		\draw[draw = orange, fill = orange] (left2)  circle[radius = 2pt];
		\draw[draw = orange, fill = orange] (top2)  circle[radius = 2pt];
		\draw[draw = orange, fill = orange] (bottom2) circle[radius = 2pt];

		\draw (top) node[anchor = south west] {$z_1$};
		\draw (top2) node[anchor = south west] {$z_1$};
		\draw (bottom) node[anchor = north west] {$z_2$};
		\draw (bottom2) node[anchor = north west] {$z_2$};
		\draw (right) node[anchor = north west] {$z$};
		\draw (left) node[anchor = north east] {$\tilde{z}$};
		\draw (right2) node[anchor = north west] {$z$};
		\draw (left2) node[anchor = north east] {$\tilde{z}$};
		\draw [->] (0.4,0) -- (1,0);
		\end{tikzpicture}
	\end{center}
	\caption{Replacing $[z_1,z_2]$ with $[z,\tilde{z}]$ increases $\deg\left(
	\tilde{z},\Nout\right) $ by 1 ( \protect\tikz\protect\draw[draw=orange,
	fill=orange]  (0,0) circle[radius = 3pt]; $\in\Nout $, \trianglecolored{red} $\in\Nin $). }
	\label{fig:angleflip}
\end{figure}

Algorithm \ref{alg:mesh-refinement} makes use of the following definitions:

\begin{enumerate}
[(i)]

\item The neighbouring nodes of a node $z\in\mathcal{N}$ are denoted by
\[
\mathcal{N}_{\operatorname{neighbours}}(z)= \{z^{\prime}\in\mathcal{N}%
\backslash\{z\} \mid[z,z^{\prime}] \in\mathcal{E}_{\Omega}\}.
\]

\item The set of edges $\mathcal{E}_{\Omega}$, where edge $[z_{1},z_{2}]$ has
been replaced by edge $[\tilde{z}_{1},\tilde{z}_{2}]$ is denoted by
\[
\mathcal{E}_{\Omega,[z_{1},z_{2}]\rightarrow\lbrack\tilde{z}_{1},\tilde{z}%
_{2}]}:=\mathcal{E}_{\Omega}\backslash\{[z_{1},z_{2}]\}\cup\{[\tilde{z}%
_{1},\tilde{z}_{2}]\}.
\]

\end{enumerate}
\begin{algorithm}[H]
	\caption{\alg\_flip}\label{alg:mesh-refinement}
	\noindent
	\hspace*{\algorithmicindent} \textbf{Input:} MOTZ\_result, \(\Nin\), \(\Nout\), \(\mathcal{N}\), \(\mathcal{E}_\Omega\) \\
	\hspace*{\algorithmicindent} \textbf{Output:} updated MOTZ\_result,
	\(\Nin\), \(\Nout\), \(\mathcal{N}\), \(\mathcal{E}_\Omega\)
	\begin{algorithmic}[1]
		\WHILE{MOTZ\_result = critical}
		\STATE Set $\mathcal{P} = \{\}$, which will hold possible modified
		meshes, together with a quality score for each mesh.
		\FORALL{\(z\in \partial\Nin\)}
		\STATE Let $z$ have exactly two neighbours in $\Nout$, denoted by $z_1, z_2 \in
		\mathcal{N}_{\operatorname{neighbours}}(z) \cap
		\Nout$
		\IF{$\mathcal{N}_{\operatorname{neighbours}}(z_1)
		\cap
			\mathcal{N}_{\operatorname{neighbours}}(z_2) \cap \Nout =
			\{\tilde{z}\}$}
		\STATE  Compute a quality score $q$ for the mesh associated with
		$\mathcal{E}_{\Omega, [z_1,z_2] \rightarrow [z,\tilde{z}]} $
		\STATE Set $\mathcal{P} = \mathcal{P} \cup \{ ({\mathcal{E}_{\Omega,
		[z_1,z_2] \rightarrow [z,\tilde{z}]} , q}) \}$
		\ENDIF
		\ENDFOR
		\IF{$\mathcal{P} = \{\}$}
		\STATE {\bf STOP}
		\ELSE
		\STATE Let $\mathcal{E}^{\max}_{\Omega, [z_1,z_2] \rightarrow
		[z,\tilde{z}]} $ be the set of edges in $\mathcal{P}$ with the highest
		quality score.
		\STATE (MOTZ\_result, $\Nin$, $\Nout$, $\mathcal{N}$, $\mathcal{E}_\Omega$) =
		\alg(\(\Nout\), $\mathcal{N}$, $\mathcal{E}^{\max}_{\Omega,
			[z_1,z_2] \rightarrow [z,\tilde{z}]}$)
		\ENDIF
		\ENDWHILE
	\end{algorithmic}
\end{algorithm}

We emphasize that the strategies described above are heuristic, e.g., only the
case of two neighbours is considered in Algorithm \ref{alg:mesh-refinement}: line 5.
However, we
expect that they are successful in the vast majority of cases where MOTZ
returns \texttt{MOTZ\_trans = false}.\newline

Figure \ref{fig:mesh_negative.pdf} shows an example of a mesh, where the
checking algorithm returns \texttt{MOTZ\_result = critical}, together with a
non-empty set
$\mathcal{N}_{\operatorname*{dof}}$ consisting of four points (see Figure
\ref{fig:mesh_negative_false} ). We apply Algorithm \ref{alg:mesh-refinement}
MOTZ\_flip, which in this
case will detect four edges that can potentially be flipped. As a simple
quality score we measure the minimal angle for each triangle in the mesh. Due
to the symmetry of the mesh, the quality score for each potential modification
suggested by MOTZ\_flip coincides. Therefore, there is no preference
concerning the choice of the edge that will be flipped in this case. Figure
\ref{fig:mesh_negative_improved} shows the modified mesh. Indeed, this
modification is sufficient in order for MOTZ to return \texttt{MOTZ\_result =
certified}.

\begin{figure}[htb]
	\noindent
	\begin{subfigure}[t]{.5\textwidth}
		\centering
		\includegraphics[width=1.0\linewidth]{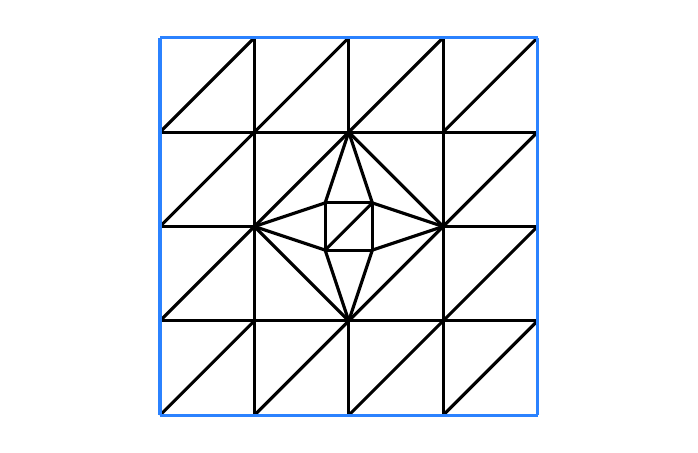}
		\caption{Original mesh}
		\label{fig:mesh_negative.pdf}
	\end{subfigure}%
	\begin{subfigure}[t]{.5\textwidth}
		\centering
		\includegraphics[width=1.0\linewidth]{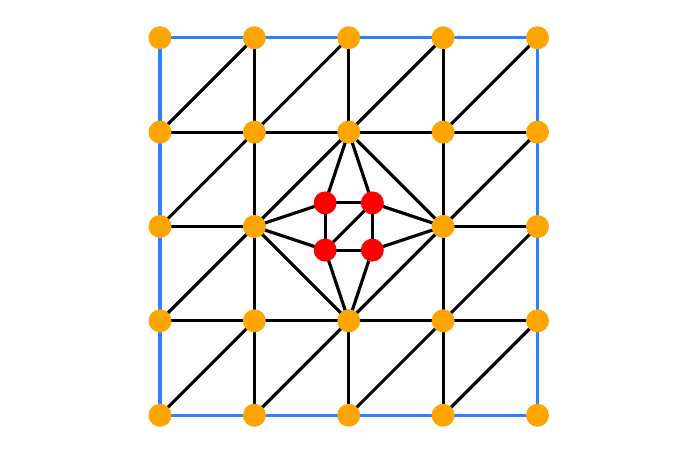}
		\caption{MOTZ returns \texttt{MOTZ\_result
		= critical}.}
		\label{fig:mesh_negative_false}
	\end{subfigure}
	\\[4mm]
	\begin{subfigure}[t]{.5\textwidth}
		\centering
		\includegraphics[width=1.0\linewidth]{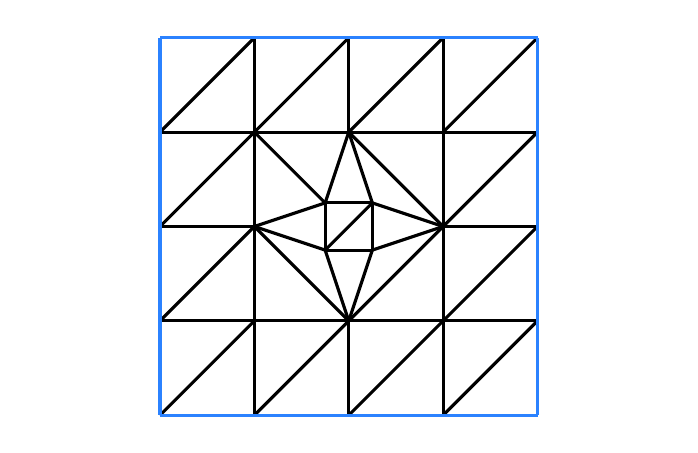}
		\caption{Modified mesh}
		\label{fig:mesh_negative_improved}
	\end{subfigure}%
	\begin{subfigure}[t]{.5\textwidth}
		\centering
		\includegraphics[width=1.0\linewidth]{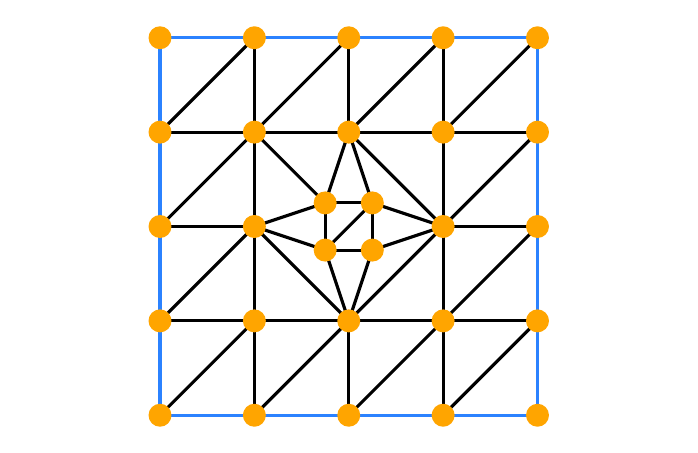}
		\caption{MOTZ returns \texttt{MOTZ\_result =
		certified}.}
		\label{fig:mesh_negative_true}
	\end{subfigure}
	\caption{ (a) Example of a mesh for which \alg\ returns
		\texttt{MOTZ\_result = critical}. (c) Modified mesh as a result of
		Algorithm
		\ref{alg:mesh-refinement} which leads to a passing checking algorithm.}
	\label{fig:mesh_negative}
\end{figure}

Below we consider the impact of mesh modification via
Algorithm \ref{alg:mesh-refinement} MOTZ\_flip by numerically calculating the reciprocal
of the discrete inf-sup constant $\beta_k$ given by
\begin{equation*}
	\beta_k :=
	\inf_{u_h \in S_{\mathcal{T}}^{1}} \,
	\sup_{v_h \in S_{\mathcal{T}}^{1}} \,
	\frac{a_k(u_h, v_h)}{\|u_h\|_{1,k,\Omega}\|v_h\|_{1,k,\Omega}},
\end{equation*}
for a modification of the mesh
considered in Section~\ref{subsection:singular_2d_example} for $\alpha = \tfrac{1}{2}$.
The $k$-weighted natural norm
$\| \cdot \|_{1,k,\Omega}$ on $H^1(\Omega)$ is given by
$\| u \|_{1,k,\Omega} := \left(\| \nabla u \|_{L^2(\Omega)}^2 + k^2 \| u
\|_{L^2(\Omega)}^2\right)^{1/2}$
for $u \in H^1(\Omega)$.
The discrete inf-sup constant $\beta_k$ can be numerically calculated via
a generalized eigenvalue problem.

For the mesh
considered in Section~\ref{subsection:singular_2d_example} with $\alpha = \tfrac{1}{2}$,
we find that $k=6$ results in a singular system matrix, see equation~\eqref{eq:critical_wavenumber}.
We inscribe the mesh considered in Section~\ref{subsection:singular_2d_example} into another quadrilateral,
see Figure~\ref{fig:singularmesh},
in order to apply Algorithm \ref{alg:mesh-refinement} MOTZ\_flip.
For the left mesh in Figure~\ref{fig:singularmesh} MOTZ returns \texttt{MOTZ\_result = critical}.
The numerical results are visualized in Figure~\ref{fig:singular_numerics}.
We observe that for $k=6$ the original mesh results in a singular system matrix,
while the modified mesh from MOTZ\_flip results in a regular one.

\begin{figure}[h]
	\begin{subfigure}[t]{.5\textwidth}
		\centering
		\includegraphics[width=0.8\linewidth]{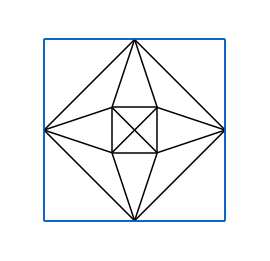}
		\label{fig:singularmesh1}
	\end{subfigure}%
	\begin{subfigure}[t]{.5\textwidth}
		\centering
		\includegraphics[width=0.8\linewidth]{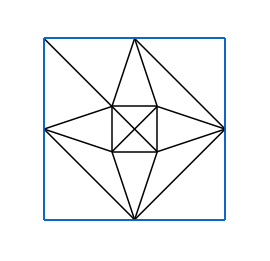}
		\label{fig:singularmesh2}
	\end{subfigure}
	\caption{For the mesh on the left, MOTZ returns \texttt{MOTZ\_result = critical}. The mesh on the right is the result of MOTZ\_flip.}
	\label{fig:singularmesh}
\end{figure}

\begin{figure}[h]
    \begin{center}
        \input{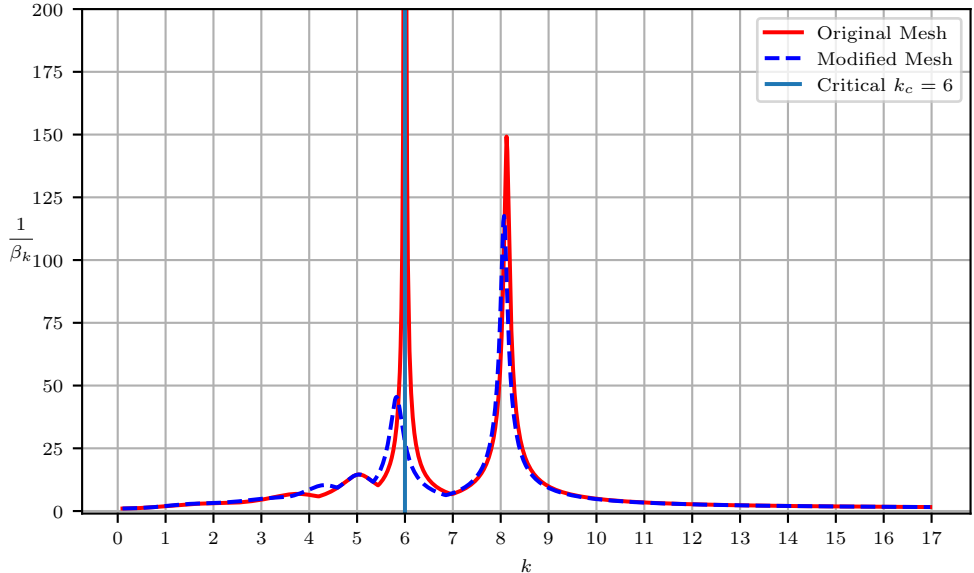}
    \end{center}
    \caption{ Plot of $k$ against $\tfrac{1}{\beta_k}$ for the meshes as in Figure~\ref{fig:singularmesh}.}
		\label{fig:singular_numerics}
\end{figure}

The algorithms MOTZ and MOTZ\_flip have been implemented in Python. The code is
available via \texttt{https://github.com/alexander-veit/MOTZ}.

\section*{Acknowledgements}
We thank Victorita Dolean, University of Strathclyde, UK,
	for valuable discussions on the topic of the paper. The first author is
	grateful for the financial support by the Austrian Science
	Fund (FWF) through the doctoral school \textit{Dissipation and dispersion
	in nonlinear PDEs} (grant W1245). The third author gratefully
acknowledges the support by the Swiss National Science Foundation under grant
no. 172803.


\def\cprime{$'$}

\end{document}

%% file: supplements/counter_example_trigs_1.tex
\definecolor{uuuuuu}{rgb}{0.26666666666666666,0.26666666666666666,0.26666666666666666}
\begin{tikzpicture}[line cap=round,line join=round,>=triangle 45,x=3cm,y=3cm]
\draw  (-1,1)-- (1,1); 
\draw  (1,1)-- (1,-1);
\draw  (1,-1)-- (-1,-1);
\draw  (-1,-1)-- (-1,1);
\draw  (-1,1)-- (-0.4,0);
\draw  (-0.4,0)-- (-1,-1);
\draw  (-1,-1)-- (0,-0.4);
\draw  (0,-0.4)-- (1,-1);
\draw  (1,-1)-- (0.4,0);
\draw  (0.4,0)-- (1,1);
\draw  (1,1)-- (0,0.4);
\draw  (0,0)-- node[below] {$\alpha$} ++ (0.35,0);
\draw  (0,0.4)-- (-1,1);
\draw  (-0.4,0)-- (0,0.4);
\draw  (0,0.4)-- (0.4,0);
\draw  (0.4,0)-- (0,-0.4);
\draw  (0,-0.4)-- (-0.4,0);
\draw  (-0.4,0)-- (0.4,0);
\draw  (0,0.4)-- (0,-0.4);
\draw [fill=black] (1,1) circle (2pt);
\draw (1,1.2) node {$P_{3}^\Gamma = (1,1)$};
\draw [fill=black] (1,-1) circle (2pt);
\draw (1,-1.2) node {$P_2^\Gamma = (1,-1)$};
\draw [fill=black] (-1,-1) circle (2pt);
\draw (-1,-1.2) node {$P_1^\Gamma = (-1,-1)$};
\draw [fill=black] (-1,1) circle (2pt);
\draw (-1,1.2) node {$P_4^\Gamma = (-1,1)$};
\draw [fill=black] (0,0.4) circle (2pt);
\draw (0,0.6) node {$P_4^\Omega$};
\draw [fill=black] (0.4,0) circle (2pt);
\draw (0.6,0) node {$P_3^\Omega$};
\draw [fill=black] (0,-0.4) circle (2pt);
\draw (0,-0.6) node {$P_2^\Omega$};
\draw [fill=black] (-0.4,0) circle (2pt);
\draw (-0.6,0) node {$P_1^\Omega$};
\draw [fill=black] (0,0) circle (2pt);
\draw (0.1,0.1) node {$P_5^\Omega$};
\end{tikzpicture}

%% file: graphics/alpha_vs_k.pgf
\begingroup%
\makeatletter%
\begin{pgfpicture}%
\pgfpathrectangle{\pgfpointorigin}{\pgfqpoint{2.800000in}{3.500000in}}%
\pgfusepath{use as bounding box, clip}%
\begin{pgfscope}%
\pgfsetbuttcap%
\pgfsetmiterjoin%
\pgfsetlinewidth{0.000000pt}%
\definecolor{currentstroke}{rgb}{1.000000,1.000000,1.000000}%
\pgfsetstrokecolor{currentstroke}%
\pgfsetstrokeopacity{0.000000}%
\pgfsetdash{}{0pt}%
\pgfpathmoveto{\pgfqpoint{0.000000in}{0.000000in}}%
\pgfpathlineto{\pgfqpoint{2.800000in}{0.000000in}}%
\pgfpathlineto{\pgfqpoint{2.800000in}{3.500000in}}%
\pgfpathlineto{\pgfqpoint{0.000000in}{3.500000in}}%
\pgfpathlineto{\pgfqpoint{0.000000in}{0.000000in}}%
\pgfpathclose%
\pgfusepath{}%
\end{pgfscope}%
\begin{pgfscope}%
\pgfsetbuttcap%
\pgfsetmiterjoin%
\definecolor{currentfill}{rgb}{1.000000,1.000000,1.000000}%
\pgfsetfillcolor{currentfill}%
\pgfsetlinewidth{0.000000pt}%
\definecolor{currentstroke}{rgb}{0.000000,0.000000,0.000000}%
\pgfsetstrokecolor{currentstroke}%
\pgfsetstrokeopacity{0.000000}%
\pgfsetdash{}{0pt}%
\pgfpathmoveto{\pgfqpoint{0.350000in}{0.437500in}}%
\pgfpathlineto{\pgfqpoint{2.520000in}{0.437500in}}%
\pgfpathlineto{\pgfqpoint{2.520000in}{3.080000in}}%
\pgfpathlineto{\pgfqpoint{0.350000in}{3.080000in}}%
\pgfpathlineto{\pgfqpoint{0.350000in}{0.437500in}}%
\pgfpathclose%
\pgfusepath{fill}%
\end{pgfscope}%
\begin{pgfscope}%
\pgfpathrectangle{\pgfqpoint{0.350000in}{0.437500in}}{\pgfqpoint{2.170000in}{2.642500in}}%
\pgfusepath{clip}%
\pgfsetrectcap%
\pgfsetroundjoin%
\pgfsetlinewidth{0.803000pt}%
\definecolor{currentstroke}{rgb}{0.690196,0.690196,0.690196}%
\pgfsetstrokecolor{currentstroke}%
\pgfsetdash{}{0pt}%
\pgfpathmoveto{\pgfqpoint{0.695227in}{0.437500in}}%
\pgfpathlineto{\pgfqpoint{0.695227in}{3.080000in}}%
\pgfusepath{stroke}%
\end{pgfscope}%
\begin{pgfscope}%
\pgfsetbuttcap%
\pgfsetroundjoin%
\definecolor{currentfill}{rgb}{0.000000,0.000000,0.000000}%
\pgfsetfillcolor{currentfill}%
\pgfsetlinewidth{0.803000pt}%
\definecolor{currentstroke}{rgb}{0.000000,0.000000,0.000000}%
\pgfsetstrokecolor{currentstroke}%
\pgfsetdash{}{0pt}%
\pgfsys@defobject{currentmarker}{\pgfqpoint{0.000000in}{-0.048611in}}{\pgfqpoint{0.000000in}{0.000000in}}{%
\pgfpathmoveto{\pgfqpoint{0.000000in}{0.000000in}}%
\pgfpathlineto{\pgfqpoint{0.000000in}{-0.048611in}}%
\pgfusepath{stroke,fill}%
}%
\begin{pgfscope}%
\pgfsys@transformshift{0.695227in}{0.437500in}%
\pgfsys@useobject{currentmarker}{}%
\end{pgfscope}%
\end{pgfscope}%
\begin{pgfscope}%
\definecolor{textcolor}{rgb}{0.000000,0.000000,0.000000}%
\pgfsetstrokecolor{textcolor}%
\pgfsetfillcolor{textcolor}%
\pgftext[x=0.695227in,y=0.340278in,,top]{\color{textcolor}\rmfamily\fontsize{7.000000}{8.400000}\selectfont \(\displaystyle {0.2}\)}%
\end{pgfscope}%
\begin{pgfscope}%
\pgfpathrectangle{\pgfqpoint{0.350000in}{0.437500in}}{\pgfqpoint{2.170000in}{2.642500in}}%
\pgfusepath{clip}%
\pgfsetrectcap%
\pgfsetroundjoin%
\pgfsetlinewidth{0.803000pt}%
\definecolor{currentstroke}{rgb}{0.690196,0.690196,0.690196}%
\pgfsetstrokecolor{currentstroke}%
\pgfsetdash{}{0pt}%
\pgfpathmoveto{\pgfqpoint{1.188409in}{0.437500in}}%
\pgfpathlineto{\pgfqpoint{1.188409in}{3.080000in}}%
\pgfusepath{stroke}%
\end{pgfscope}%
\begin{pgfscope}%
\pgfsetbuttcap%
\pgfsetroundjoin%
\definecolor{currentfill}{rgb}{0.000000,0.000000,0.000000}%
\pgfsetfillcolor{currentfill}%
\pgfsetlinewidth{0.803000pt}%
\definecolor{currentstroke}{rgb}{0.000000,0.000000,0.000000}%
\pgfsetstrokecolor{currentstroke}%
\pgfsetdash{}{0pt}%
\pgfsys@defobject{currentmarker}{\pgfqpoint{0.000000in}{-0.048611in}}{\pgfqpoint{0.000000in}{0.000000in}}{%
\pgfpathmoveto{\pgfqpoint{0.000000in}{0.000000in}}%
\pgfpathlineto{\pgfqpoint{0.000000in}{-0.048611in}}%
\pgfusepath{stroke,fill}%
}%
\begin{pgfscope}%
\pgfsys@transformshift{1.188409in}{0.437500in}%
\pgfsys@useobject{currentmarker}{}%
\end{pgfscope}%
\end{pgfscope}%
\begin{pgfscope}%
\definecolor{textcolor}{rgb}{0.000000,0.000000,0.000000}%
\pgfsetstrokecolor{textcolor}%
\pgfsetfillcolor{textcolor}%
\pgftext[x=1.188409in,y=0.340278in,,top]{\color{textcolor}\rmfamily\fontsize{7.000000}{8.400000}\selectfont \(\displaystyle {0.4}\)}%
\end{pgfscope}%
\begin{pgfscope}%
\pgfpathrectangle{\pgfqpoint{0.350000in}{0.437500in}}{\pgfqpoint{2.170000in}{2.642500in}}%
\pgfusepath{clip}%
\pgfsetrectcap%
\pgfsetroundjoin%
\pgfsetlinewidth{0.803000pt}%
\definecolor{currentstroke}{rgb}{0.690196,0.690196,0.690196}%
\pgfsetstrokecolor{currentstroke}%
\pgfsetdash{}{0pt}%
\pgfpathmoveto{\pgfqpoint{1.681591in}{0.437500in}}%
\pgfpathlineto{\pgfqpoint{1.681591in}{3.080000in}}%
\pgfusepath{stroke}%
\end{pgfscope}%
\begin{pgfscope}%
\pgfsetbuttcap%
\pgfsetroundjoin%
\definecolor{currentfill}{rgb}{0.000000,0.000000,0.000000}%
\pgfsetfillcolor{currentfill}%
\pgfsetlinewidth{0.803000pt}%
\definecolor{currentstroke}{rgb}{0.000000,0.000000,0.000000}%
\pgfsetstrokecolor{currentstroke}%
\pgfsetdash{}{0pt}%
\pgfsys@defobject{currentmarker}{\pgfqpoint{0.000000in}{-0.048611in}}{\pgfqpoint{0.000000in}{0.000000in}}{%
\pgfpathmoveto{\pgfqpoint{0.000000in}{0.000000in}}%
\pgfpathlineto{\pgfqpoint{0.000000in}{-0.048611in}}%
\pgfusepath{stroke,fill}%
}%
\begin{pgfscope}%
\pgfsys@transformshift{1.681591in}{0.437500in}%
\pgfsys@useobject{currentmarker}{}%
\end{pgfscope}%
\end{pgfscope}%
\begin{pgfscope}%
\definecolor{textcolor}{rgb}{0.000000,0.000000,0.000000}%
\pgfsetstrokecolor{textcolor}%
\pgfsetfillcolor{textcolor}%
\pgftext[x=1.681591in,y=0.340278in,,top]{\color{textcolor}\rmfamily\fontsize{7.000000}{8.400000}\selectfont \(\displaystyle {0.6}\)}%
\end{pgfscope}%
\begin{pgfscope}%
\pgfpathrectangle{\pgfqpoint{0.350000in}{0.437500in}}{\pgfqpoint{2.170000in}{2.642500in}}%
\pgfusepath{clip}%
\pgfsetrectcap%
\pgfsetroundjoin%
\pgfsetlinewidth{0.803000pt}%
\definecolor{currentstroke}{rgb}{0.690196,0.690196,0.690196}%
\pgfsetstrokecolor{currentstroke}%
\pgfsetdash{}{0pt}%
\pgfpathmoveto{\pgfqpoint{2.174773in}{0.437500in}}%
\pgfpathlineto{\pgfqpoint{2.174773in}{3.080000in}}%
\pgfusepath{stroke}%
\end{pgfscope}%
\begin{pgfscope}%
\pgfsetbuttcap%
\pgfsetroundjoin%
\definecolor{currentfill}{rgb}{0.000000,0.000000,0.000000}%
\pgfsetfillcolor{currentfill}%
\pgfsetlinewidth{0.803000pt}%
\definecolor{currentstroke}{rgb}{0.000000,0.000000,0.000000}%
\pgfsetstrokecolor{currentstroke}%
\pgfsetdash{}{0pt}%
\pgfsys@defobject{currentmarker}{\pgfqpoint{0.000000in}{-0.048611in}}{\pgfqpoint{0.000000in}{0.000000in}}{%
\pgfpathmoveto{\pgfqpoint{0.000000in}{0.000000in}}%
\pgfpathlineto{\pgfqpoint{0.000000in}{-0.048611in}}%
\pgfusepath{stroke,fill}%
}%
\begin{pgfscope}%
\pgfsys@transformshift{2.174773in}{0.437500in}%
\pgfsys@useobject{currentmarker}{}%
\end{pgfscope}%
\end{pgfscope}%
\begin{pgfscope}%
\definecolor{textcolor}{rgb}{0.000000,0.000000,0.000000}%
\pgfsetstrokecolor{textcolor}%
\pgfsetfillcolor{textcolor}%
\pgftext[x=2.174773in,y=0.340278in,,top]{\color{textcolor}\rmfamily\fontsize{7.000000}{8.400000}\selectfont \(\displaystyle {0.8}\)}%
\end{pgfscope}%
\begin{pgfscope}%
\definecolor{textcolor}{rgb}{0.000000,0.000000,0.000000}%
\pgfsetstrokecolor{textcolor}%
\pgfsetfillcolor{textcolor}%
\pgftext[x=1.435000in,y=0.198303in,,top]{\color{textcolor}\rmfamily\fontsize{7.000000}{8.400000}\selectfont \(\displaystyle \alpha\)}%
\end{pgfscope}%
\begin{pgfscope}%
\pgfpathrectangle{\pgfqpoint{0.350000in}{0.437500in}}{\pgfqpoint{2.170000in}{2.642500in}}%
\pgfusepath{clip}%
\pgfsetrectcap%
\pgfsetroundjoin%
\pgfsetlinewidth{0.803000pt}%
\definecolor{currentstroke}{rgb}{0.690196,0.690196,0.690196}%
\pgfsetstrokecolor{currentstroke}%
\pgfsetdash{}{0pt}%
\pgfpathmoveto{\pgfqpoint{0.350000in}{0.596478in}}%
\pgfpathlineto{\pgfqpoint{2.520000in}{0.596478in}}%
\pgfusepath{stroke}%
\end{pgfscope}%
\begin{pgfscope}%
\pgfsetbuttcap%
\pgfsetroundjoin%
\definecolor{currentfill}{rgb}{0.000000,0.000000,0.000000}%
\pgfsetfillcolor{currentfill}%
\pgfsetlinewidth{0.803000pt}%
\definecolor{currentstroke}{rgb}{0.000000,0.000000,0.000000}%
\pgfsetstrokecolor{currentstroke}%
\pgfsetdash{}{0pt}%
\pgfsys@defobject{currentmarker}{\pgfqpoint{-0.048611in}{0.000000in}}{\pgfqpoint{-0.000000in}{0.000000in}}{%
\pgfpathmoveto{\pgfqpoint{-0.000000in}{0.000000in}}%
\pgfpathlineto{\pgfqpoint{-0.048611in}{0.000000in}}%
\pgfusepath{stroke,fill}%
}%
\begin{pgfscope}%
\pgfsys@transformshift{0.350000in}{0.596478in}%
\pgfsys@useobject{currentmarker}{}%
\end{pgfscope}%
\end{pgfscope}%
\begin{pgfscope}%
\definecolor{textcolor}{rgb}{0.000000,0.000000,0.000000}%
\pgfsetstrokecolor{textcolor}%
\pgfsetfillcolor{textcolor}%
\pgftext[x=0.197415in, y=0.562721in, left, base]{\color{textcolor}\rmfamily\fontsize{7.000000}{8.400000}\selectfont \(\displaystyle {6}\)}%
\end{pgfscope}%
\begin{pgfscope}%
\pgfpathrectangle{\pgfqpoint{0.350000in}{0.437500in}}{\pgfqpoint{2.170000in}{2.642500in}}%
\pgfusepath{clip}%
\pgfsetrectcap%
\pgfsetroundjoin%
\pgfsetlinewidth{0.803000pt}%
\definecolor{currentstroke}{rgb}{0.690196,0.690196,0.690196}%
\pgfsetstrokecolor{currentstroke}%
\pgfsetdash{}{0pt}%
\pgfpathmoveto{\pgfqpoint{0.350000in}{1.046255in}}%
\pgfpathlineto{\pgfqpoint{2.520000in}{1.046255in}}%
\pgfusepath{stroke}%
\end{pgfscope}%
\begin{pgfscope}%
\pgfsetbuttcap%
\pgfsetroundjoin%
\definecolor{currentfill}{rgb}{0.000000,0.000000,0.000000}%
\pgfsetfillcolor{currentfill}%
\pgfsetlinewidth{0.803000pt}%
\definecolor{currentstroke}{rgb}{0.000000,0.000000,0.000000}%
\pgfsetstrokecolor{currentstroke}%
\pgfsetdash{}{0pt}%
\pgfsys@defobject{currentmarker}{\pgfqpoint{-0.048611in}{0.000000in}}{\pgfqpoint{-0.000000in}{0.000000in}}{%
\pgfpathmoveto{\pgfqpoint{-0.000000in}{0.000000in}}%
\pgfpathlineto{\pgfqpoint{-0.048611in}{0.000000in}}%
\pgfusepath{stroke,fill}%
}%
\begin{pgfscope}%
\pgfsys@transformshift{0.350000in}{1.046255in}%
\pgfsys@useobject{currentmarker}{}%
\end{pgfscope}%
\end{pgfscope}%
\begin{pgfscope}%
\definecolor{textcolor}{rgb}{0.000000,0.000000,0.000000}%
\pgfsetstrokecolor{textcolor}%
\pgfsetfillcolor{textcolor}%
\pgftext[x=0.197415in, y=1.012497in, left, base]{\color{textcolor}\rmfamily\fontsize{7.000000}{8.400000}\selectfont \(\displaystyle {7}\)}%
\end{pgfscope}%
\begin{pgfscope}%
\pgfpathrectangle{\pgfqpoint{0.350000in}{0.437500in}}{\pgfqpoint{2.170000in}{2.642500in}}%
\pgfusepath{clip}%
\pgfsetrectcap%
\pgfsetroundjoin%
\pgfsetlinewidth{0.803000pt}%
\definecolor{currentstroke}{rgb}{0.690196,0.690196,0.690196}%
\pgfsetstrokecolor{currentstroke}%
\pgfsetdash{}{0pt}%
\pgfpathmoveto{\pgfqpoint{0.350000in}{1.496031in}}%
\pgfpathlineto{\pgfqpoint{2.520000in}{1.496031in}}%
\pgfusepath{stroke}%
\end{pgfscope}%
\begin{pgfscope}%
\pgfsetbuttcap%
\pgfsetroundjoin%
\definecolor{currentfill}{rgb}{0.000000,0.000000,0.000000}%
\pgfsetfillcolor{currentfill}%
\pgfsetlinewidth{0.803000pt}%
\definecolor{currentstroke}{rgb}{0.000000,0.000000,0.000000}%
\pgfsetstrokecolor{currentstroke}%
\pgfsetdash{}{0pt}%
\pgfsys@defobject{currentmarker}{\pgfqpoint{-0.048611in}{0.000000in}}{\pgfqpoint{-0.000000in}{0.000000in}}{%
\pgfpathmoveto{\pgfqpoint{-0.000000in}{0.000000in}}%
\pgfpathlineto{\pgfqpoint{-0.048611in}{0.000000in}}%
\pgfusepath{stroke,fill}%
}%
\begin{pgfscope}%
\pgfsys@transformshift{0.350000in}{1.496031in}%
\pgfsys@useobject{currentmarker}{}%
\end{pgfscope}%
\end{pgfscope}%
\begin{pgfscope}%
\definecolor{textcolor}{rgb}{0.000000,0.000000,0.000000}%
\pgfsetstrokecolor{textcolor}%
\pgfsetfillcolor{textcolor}%
\pgftext[x=0.197415in, y=1.462273in, left, base]{\color{textcolor}\rmfamily\fontsize{7.000000}{8.400000}\selectfont \(\displaystyle {8}\)}%
\end{pgfscope}%
\begin{pgfscope}%
\pgfpathrectangle{\pgfqpoint{0.350000in}{0.437500in}}{\pgfqpoint{2.170000in}{2.642500in}}%
\pgfusepath{clip}%
\pgfsetrectcap%
\pgfsetroundjoin%
\pgfsetlinewidth{0.803000pt}%
\definecolor{currentstroke}{rgb}{0.690196,0.690196,0.690196}%
\pgfsetstrokecolor{currentstroke}%
\pgfsetdash{}{0pt}%
\pgfpathmoveto{\pgfqpoint{0.350000in}{1.945808in}}%
\pgfpathlineto{\pgfqpoint{2.520000in}{1.945808in}}%
\pgfusepath{stroke}%
\end{pgfscope}%
\begin{pgfscope}%
\pgfsetbuttcap%
\pgfsetroundjoin%
\definecolor{currentfill}{rgb}{0.000000,0.000000,0.000000}%
\pgfsetfillcolor{currentfill}%
\pgfsetlinewidth{0.803000pt}%
\definecolor{currentstroke}{rgb}{0.000000,0.000000,0.000000}%
\pgfsetstrokecolor{currentstroke}%
\pgfsetdash{}{0pt}%
\pgfsys@defobject{currentmarker}{\pgfqpoint{-0.048611in}{0.000000in}}{\pgfqpoint{-0.000000in}{0.000000in}}{%
\pgfpathmoveto{\pgfqpoint{-0.000000in}{0.000000in}}%
\pgfpathlineto{\pgfqpoint{-0.048611in}{0.000000in}}%
\pgfusepath{stroke,fill}%
}%
\begin{pgfscope}%
\pgfsys@transformshift{0.350000in}{1.945808in}%
\pgfsys@useobject{currentmarker}{}%
\end{pgfscope}%
\end{pgfscope}%
\begin{pgfscope}%
\definecolor{textcolor}{rgb}{0.000000,0.000000,0.000000}%
\pgfsetstrokecolor{textcolor}%
\pgfsetfillcolor{textcolor}%
\pgftext[x=0.197415in, y=1.912050in, left, base]{\color{textcolor}\rmfamily\fontsize{7.000000}{8.400000}\selectfont \(\displaystyle {9}\)}%
\end{pgfscope}%
\begin{pgfscope}%
\pgfpathrectangle{\pgfqpoint{0.350000in}{0.437500in}}{\pgfqpoint{2.170000in}{2.642500in}}%
\pgfusepath{clip}%
\pgfsetrectcap%
\pgfsetroundjoin%
\pgfsetlinewidth{0.803000pt}%
\definecolor{currentstroke}{rgb}{0.690196,0.690196,0.690196}%
\pgfsetstrokecolor{currentstroke}%
\pgfsetdash{}{0pt}%
\pgfpathmoveto{\pgfqpoint{0.350000in}{2.395584in}}%
\pgfpathlineto{\pgfqpoint{2.520000in}{2.395584in}}%
\pgfusepath{stroke}%
\end{pgfscope}%
\begin{pgfscope}%
\pgfsetbuttcap%
\pgfsetroundjoin%
\definecolor{currentfill}{rgb}{0.000000,0.000000,0.000000}%
\pgfsetfillcolor{currentfill}%
\pgfsetlinewidth{0.803000pt}%
\definecolor{currentstroke}{rgb}{0.000000,0.000000,0.000000}%
\pgfsetstrokecolor{currentstroke}%
\pgfsetdash{}{0pt}%
\pgfsys@defobject{currentmarker}{\pgfqpoint{-0.048611in}{0.000000in}}{\pgfqpoint{-0.000000in}{0.000000in}}{%
\pgfpathmoveto{\pgfqpoint{-0.000000in}{0.000000in}}%
\pgfpathlineto{\pgfqpoint{-0.048611in}{0.000000in}}%
\pgfusepath{stroke,fill}%
}%
\begin{pgfscope}%
\pgfsys@transformshift{0.350000in}{2.395584in}%
\pgfsys@useobject{currentmarker}{}%
\end{pgfscope}%
\end{pgfscope}%
\begin{pgfscope}%
\definecolor{textcolor}{rgb}{0.000000,0.000000,0.000000}%
\pgfsetstrokecolor{textcolor}%
\pgfsetfillcolor{textcolor}%
\pgftext[x=0.142052in, y=2.361826in, left, base]{\color{textcolor}\rmfamily\fontsize{7.000000}{8.400000}\selectfont \(\displaystyle {10}\)}%
\end{pgfscope}%
\begin{pgfscope}%
\pgfpathrectangle{\pgfqpoint{0.350000in}{0.437500in}}{\pgfqpoint{2.170000in}{2.642500in}}%
\pgfusepath{clip}%
\pgfsetrectcap%
\pgfsetroundjoin%
\pgfsetlinewidth{0.803000pt}%
\definecolor{currentstroke}{rgb}{0.690196,0.690196,0.690196}%
\pgfsetstrokecolor{currentstroke}%
\pgfsetdash{}{0pt}%
\pgfpathmoveto{\pgfqpoint{0.350000in}{2.845360in}}%
\pgfpathlineto{\pgfqpoint{2.520000in}{2.845360in}}%
\pgfusepath{stroke}%
\end{pgfscope}%
\begin{pgfscope}%
\pgfsetbuttcap%
\pgfsetroundjoin%
\definecolor{currentfill}{rgb}{0.000000,0.000000,0.000000}%
\pgfsetfillcolor{currentfill}%
\pgfsetlinewidth{0.803000pt}%
\definecolor{currentstroke}{rgb}{0.000000,0.000000,0.000000}%
\pgfsetstrokecolor{currentstroke}%
\pgfsetdash{}{0pt}%
\pgfsys@defobject{currentmarker}{\pgfqpoint{-0.048611in}{0.000000in}}{\pgfqpoint{-0.000000in}{0.000000in}}{%
\pgfpathmoveto{\pgfqpoint{-0.000000in}{0.000000in}}%
\pgfpathlineto{\pgfqpoint{-0.048611in}{0.000000in}}%
\pgfusepath{stroke,fill}%
}%
\begin{pgfscope}%
\pgfsys@transformshift{0.350000in}{2.845360in}%
\pgfsys@useobject{currentmarker}{}%
\end{pgfscope}%
\end{pgfscope}%
\begin{pgfscope}%
\definecolor{textcolor}{rgb}{0.000000,0.000000,0.000000}%
\pgfsetstrokecolor{textcolor}%
\pgfsetfillcolor{textcolor}%
\pgftext[x=0.142052in, y=2.811603in, left, base]{\color{textcolor}\rmfamily\fontsize{7.000000}{8.400000}\selectfont \(\displaystyle {11}\)}%
\end{pgfscope}%
\begin{pgfscope}%
\definecolor{textcolor}{rgb}{0.000000,0.000000,0.000000}%
\pgfsetstrokecolor{textcolor}%
\pgfsetfillcolor{textcolor}%
\pgftext[x=0.086496in,y=1.758750in,,bottom]{\color{textcolor}\rmfamily\fontsize{7.000000}{8.400000}\selectfont \(\displaystyle k\)}%
\end{pgfscope}%
\begin{pgfscope}%
\pgfpathrectangle{\pgfqpoint{0.350000in}{0.437500in}}{\pgfqpoint{2.170000in}{2.642500in}}%
\pgfusepath{clip}%
\pgfsetrectcap%
\pgfsetroundjoin%
\pgfsetlinewidth{1.505625pt}%
\definecolor{currentstroke}{rgb}{1.000000,0.000000,0.000000}%
\pgfsetstrokecolor{currentstroke}%
\pgfsetdash{}{0pt}%
\pgfpathmoveto{\pgfqpoint{0.448636in}{2.959886in}}%
\pgfpathlineto{\pgfqpoint{0.468563in}{2.778570in}}%
\pgfpathlineto{\pgfqpoint{0.488489in}{2.617107in}}%
\pgfpathlineto{\pgfqpoint{0.508416in}{2.472216in}}%
\pgfpathlineto{\pgfqpoint{0.528343in}{2.341329in}}%
\pgfpathlineto{\pgfqpoint{0.548269in}{2.222407in}}%
\pgfpathlineto{\pgfqpoint{0.568196in}{2.113801in}}%
\pgfpathlineto{\pgfqpoint{0.588122in}{2.014169in}}%
\pgfpathlineto{\pgfqpoint{0.608049in}{1.922402in}}%
\pgfpathlineto{\pgfqpoint{0.627975in}{1.837574in}}%
\pgfpathlineto{\pgfqpoint{0.647902in}{1.758909in}}%
\pgfpathlineto{\pgfqpoint{0.667828in}{1.685746in}}%
\pgfpathlineto{\pgfqpoint{0.687755in}{1.617522in}}%
\pgfpathlineto{\pgfqpoint{0.707681in}{1.553753in}}%
\pgfpathlineto{\pgfqpoint{0.727608in}{1.494020in}}%
\pgfpathlineto{\pgfqpoint{0.747534in}{1.437959in}}%
\pgfpathlineto{\pgfqpoint{0.767461in}{1.385251in}}%
\pgfpathlineto{\pgfqpoint{0.787388in}{1.335617in}}%
\pgfpathlineto{\pgfqpoint{0.807314in}{1.288809in}}%
\pgfpathlineto{\pgfqpoint{0.827241in}{1.244609in}}%
\pgfpathlineto{\pgfqpoint{0.847167in}{1.202822in}}%
\pgfpathlineto{\pgfqpoint{0.867094in}{1.163275in}}%
\pgfpathlineto{\pgfqpoint{0.887020in}{1.125814in}}%
\pgfpathlineto{\pgfqpoint{0.906947in}{1.090298in}}%
\pgfpathlineto{\pgfqpoint{0.926873in}{1.056602in}}%
\pgfpathlineto{\pgfqpoint{0.946800in}{1.024613in}}%
\pgfpathlineto{\pgfqpoint{0.966726in}{0.994229in}}%
\pgfpathlineto{\pgfqpoint{0.986653in}{0.965357in}}%
\pgfpathlineto{\pgfqpoint{1.006579in}{0.937913in}}%
\pgfpathlineto{\pgfqpoint{1.026506in}{0.911822in}}%
\pgfpathlineto{\pgfqpoint{1.046433in}{0.887012in}}%
\pgfpathlineto{\pgfqpoint{1.066359in}{0.863421in}}%
\pgfpathlineto{\pgfqpoint{1.086286in}{0.840991in}}%
\pgfpathlineto{\pgfqpoint{1.106212in}{0.819669in}}%
\pgfpathlineto{\pgfqpoint{1.126139in}{0.799407in}}%
\pgfpathlineto{\pgfqpoint{1.146065in}{0.780161in}}%
\pgfpathlineto{\pgfqpoint{1.165992in}{0.761890in}}%
\pgfpathlineto{\pgfqpoint{1.185918in}{0.744558in}}%
\pgfpathlineto{\pgfqpoint{1.205845in}{0.728131in}}%
\pgfpathlineto{\pgfqpoint{1.225771in}{0.712578in}}%
\pgfpathlineto{\pgfqpoint{1.245698in}{0.697873in}}%
\pgfpathlineto{\pgfqpoint{1.265624in}{0.683989in}}%
\pgfpathlineto{\pgfqpoint{1.285551in}{0.670904in}}%
\pgfpathlineto{\pgfqpoint{1.305478in}{0.658597in}}%
\pgfpathlineto{\pgfqpoint{1.325404in}{0.647049in}}%
\pgfpathlineto{\pgfqpoint{1.345331in}{0.636245in}}%
\pgfpathlineto{\pgfqpoint{1.365257in}{0.626170in}}%
\pgfpathlineto{\pgfqpoint{1.385184in}{0.616810in}}%
\pgfpathlineto{\pgfqpoint{1.405110in}{0.608156in}}%
\pgfpathlineto{\pgfqpoint{1.425037in}{0.600199in}}%
\pgfpathlineto{\pgfqpoint{1.444963in}{0.592929in}}%
\pgfpathlineto{\pgfqpoint{1.464890in}{0.586343in}}%
\pgfpathlineto{\pgfqpoint{1.484816in}{0.580434in}}%
\pgfpathlineto{\pgfqpoint{1.504743in}{0.575202in}}%
\pgfpathlineto{\pgfqpoint{1.524669in}{0.570645in}}%
\pgfpathlineto{\pgfqpoint{1.544596in}{0.566762in}}%
\pgfpathlineto{\pgfqpoint{1.564522in}{0.563558in}}%
\pgfpathlineto{\pgfqpoint{1.584449in}{0.561035in}}%
\pgfpathlineto{\pgfqpoint{1.604376in}{0.559199in}}%
\pgfpathlineto{\pgfqpoint{1.624302in}{0.558057in}}%
\pgfpathlineto{\pgfqpoint{1.644229in}{0.557618in}}%
\pgfpathlineto{\pgfqpoint{1.664155in}{0.557895in}}%
\pgfpathlineto{\pgfqpoint{1.684082in}{0.558899in}}%
\pgfpathlineto{\pgfqpoint{1.704008in}{0.560645in}}%
\pgfpathlineto{\pgfqpoint{1.723935in}{0.563152in}}%
\pgfpathlineto{\pgfqpoint{1.743861in}{0.566439in}}%
\pgfpathlineto{\pgfqpoint{1.763788in}{0.570529in}}%
\pgfpathlineto{\pgfqpoint{1.783714in}{0.575446in}}%
\pgfpathlineto{\pgfqpoint{1.803641in}{0.581220in}}%
\pgfpathlineto{\pgfqpoint{1.823567in}{0.587880in}}%
\pgfpathlineto{\pgfqpoint{1.843494in}{0.595464in}}%
\pgfpathlineto{\pgfqpoint{1.863421in}{0.604008in}}%
\pgfpathlineto{\pgfqpoint{1.883347in}{0.613558in}}%
\pgfpathlineto{\pgfqpoint{1.903274in}{0.624162in}}%
\pgfpathlineto{\pgfqpoint{1.923200in}{0.635872in}}%
\pgfpathlineto{\pgfqpoint{1.943127in}{0.648749in}}%
\pgfpathlineto{\pgfqpoint{1.963053in}{0.662858in}}%
\pgfpathlineto{\pgfqpoint{1.982980in}{0.678274in}}%
\pgfpathlineto{\pgfqpoint{2.002906in}{0.695079in}}%
\pgfpathlineto{\pgfqpoint{2.022833in}{0.713365in}}%
\pgfpathlineto{\pgfqpoint{2.042759in}{0.733235in}}%
\pgfpathlineto{\pgfqpoint{2.062686in}{0.754805in}}%
\pgfpathlineto{\pgfqpoint{2.082612in}{0.778205in}}%
\pgfpathlineto{\pgfqpoint{2.102539in}{0.803581in}}%
\pgfpathlineto{\pgfqpoint{2.122466in}{0.831102in}}%
\pgfpathlineto{\pgfqpoint{2.142392in}{0.860955in}}%
\pgfpathlineto{\pgfqpoint{2.162319in}{0.893357in}}%
\pgfpathlineto{\pgfqpoint{2.182245in}{0.928556in}}%
\pgfpathlineto{\pgfqpoint{2.202172in}{0.966839in}}%
\pgfpathlineto{\pgfqpoint{2.222098in}{1.008535in}}%
\pgfpathlineto{\pgfqpoint{2.242025in}{1.054031in}}%
\pgfpathlineto{\pgfqpoint{2.261951in}{1.103779in}}%
\pgfpathlineto{\pgfqpoint{2.281878in}{1.158315in}}%
\pgfpathlineto{\pgfqpoint{2.301804in}{1.218275in}}%
\pgfpathlineto{\pgfqpoint{2.321731in}{1.284426in}}%
\pgfpathlineto{\pgfqpoint{2.341657in}{1.357695in}}%
\pgfpathlineto{\pgfqpoint{2.361584in}{1.439223in}}%
\pgfpathlineto{\pgfqpoint{2.381511in}{1.530424in}}%
\pgfpathlineto{\pgfqpoint{2.401437in}{1.633081in}}%
\pgfpathlineto{\pgfqpoint{2.421364in}{1.749475in}}%
\pgfusepath{stroke}%
\end{pgfscope}%
\begin{pgfscope}%
\pgfpathrectangle{\pgfqpoint{0.350000in}{0.437500in}}{\pgfqpoint{2.170000in}{2.642500in}}%
\pgfusepath{clip}%
\pgfsetbuttcap%
\pgfsetroundjoin%
\pgfsetlinewidth{1.505625pt}%
\definecolor{currentstroke}{rgb}{0.121569,0.466667,0.705882}%
\pgfsetstrokecolor{currentstroke}%
\pgfsetdash{}{0pt}%
\pgfpathmoveto{\pgfqpoint{0.448636in}{0.557614in}}%
\pgfpathlineto{\pgfqpoint{2.421364in}{0.557614in}}%
\pgfusepath{stroke}%
\end{pgfscope}%
\begin{pgfscope}%
\pgfsetrectcap%
\pgfsetmiterjoin%
\pgfsetlinewidth{0.803000pt}%
\definecolor{currentstroke}{rgb}{0.000000,0.000000,0.000000}%
\pgfsetstrokecolor{currentstroke}%
\pgfsetdash{}{0pt}%
\pgfpathmoveto{\pgfqpoint{0.350000in}{0.437500in}}%
\pgfpathlineto{\pgfqpoint{0.350000in}{3.080000in}}%
\pgfusepath{stroke}%
\end{pgfscope}%
\begin{pgfscope}%
\pgfsetrectcap%
\pgfsetmiterjoin%
\pgfsetlinewidth{0.803000pt}%
\definecolor{currentstroke}{rgb}{0.000000,0.000000,0.000000}%
\pgfsetstrokecolor{currentstroke}%
\pgfsetdash{}{0pt}%
\pgfpathmoveto{\pgfqpoint{2.520000in}{0.437500in}}%
\pgfpathlineto{\pgfqpoint{2.520000in}{3.080000in}}%
\pgfusepath{stroke}%
\end{pgfscope}%
\begin{pgfscope}%
\pgfsetrectcap%
\pgfsetmiterjoin%
\pgfsetlinewidth{0.803000pt}%
\definecolor{currentstroke}{rgb}{0.000000,0.000000,0.000000}%
\pgfsetstrokecolor{currentstroke}%
\pgfsetdash{}{0pt}%
\pgfpathmoveto{\pgfqpoint{0.350000in}{0.437500in}}%
\pgfpathlineto{\pgfqpoint{2.520000in}{0.437500in}}%
\pgfusepath{stroke}%
\end{pgfscope}%
\begin{pgfscope}%
\pgfsetrectcap%
\pgfsetmiterjoin%
\pgfsetlinewidth{0.803000pt}%
\definecolor{currentstroke}{rgb}{0.000000,0.000000,0.000000}%
\pgfsetstrokecolor{currentstroke}%
\pgfsetdash{}{0pt}%
\pgfpathmoveto{\pgfqpoint{0.350000in}{3.080000in}}%
\pgfpathlineto{\pgfqpoint{2.520000in}{3.080000in}}%
\pgfusepath{stroke}%
\end{pgfscope}%
\begin{pgfscope}%
\pgfsetbuttcap%
\pgfsetmiterjoin%
\definecolor{currentfill}{rgb}{1.000000,1.000000,1.000000}%
\pgfsetfillcolor{currentfill}%
\pgfsetfillopacity{0.800000}%
\pgfsetlinewidth{1.003750pt}%
\definecolor{currentstroke}{rgb}{0.800000,0.800000,0.800000}%
\pgfsetstrokecolor{currentstroke}%
\pgfsetstrokeopacity{0.800000}%
\pgfsetdash{}{0pt}%
\pgfpathmoveto{\pgfqpoint{1.466598in}{2.645817in}}%
\pgfpathlineto{\pgfqpoint{2.451944in}{2.645817in}}%
\pgfpathquadraticcurveto{\pgfqpoint{2.471389in}{2.645817in}}{\pgfqpoint{2.471389in}{2.665261in}}%
\pgfpathlineto{\pgfqpoint{2.471389in}{3.011944in}}%
\pgfpathquadraticcurveto{\pgfqpoint{2.471389in}{3.031389in}}{\pgfqpoint{2.451944in}{3.031389in}}%
\pgfpathlineto{\pgfqpoint{1.466598in}{3.031389in}}%
\pgfpathquadraticcurveto{\pgfqpoint{1.447154in}{3.031389in}}{\pgfqpoint{1.447154in}{3.011944in}}%
\pgfpathlineto{\pgfqpoint{1.447154in}{2.665261in}}%
\pgfpathquadraticcurveto{\pgfqpoint{1.447154in}{2.645817in}}{\pgfqpoint{1.466598in}{2.645817in}}%
\pgfpathlineto{\pgfqpoint{1.466598in}{2.645817in}}%
\pgfpathclose%
\pgfusepath{stroke,fill}%
\end{pgfscope}%
\begin{pgfscope}%
\pgfsetrectcap%
\pgfsetroundjoin%
\pgfsetlinewidth{1.505625pt}%
\definecolor{currentstroke}{rgb}{1.000000,0.000000,0.000000}%
\pgfsetstrokecolor{currentstroke}%
\pgfsetdash{}{0pt}%
\pgfpathmoveto{\pgfqpoint{1.486043in}{2.958472in}}%
\pgfpathlineto{\pgfqpoint{1.583265in}{2.958472in}}%
\pgfpathlineto{\pgfqpoint{1.680487in}{2.958472in}}%
\pgfusepath{stroke}%
\end{pgfscope}%
\begin{pgfscope}%
\definecolor{textcolor}{rgb}{0.000000,0.000000,0.000000}%
\pgfsetstrokecolor{textcolor}%
\pgfsetfillcolor{textcolor}%
\pgftext[x=1.758265in,y=2.924444in,left,base]{\color{textcolor}\rmfamily\fontsize{7.000000}{8.400000}\selectfont Critical \(\displaystyle k_c\)}%
\end{pgfscope}%
\begin{pgfscope}%
\pgfsetbuttcap%
\pgfsetroundjoin%
\pgfsetlinewidth{1.505625pt}%
\definecolor{currentstroke}{rgb}{0.121569,0.466667,0.705882}%
\pgfsetstrokecolor{currentstroke}%
\pgfsetdash{}{0pt}%
\pgfpathmoveto{\pgfqpoint{1.486043in}{2.755917in}}%
\pgfpathlineto{\pgfqpoint{1.680487in}{2.755917in}}%
\pgfusepath{stroke}%
\end{pgfscope}%
\begin{pgfscope}%
\definecolor{textcolor}{rgb}{0.000000,0.000000,0.000000}%
\pgfsetstrokecolor{textcolor}%
\pgfsetfillcolor{textcolor}%
\pgftext[x=1.758265in,y=2.721889in,left,base]{\color{textcolor}\rmfamily\fontsize{7.000000}{8.400000}\selectfont \(\displaystyle \sqrt{6(3+2\sqrt{2})}\)}%
\end{pgfscope}%
\end{pgfpicture}%
\makeatother%
\endgroup%

%% file: supplements/counter_example_trigs_makro.tex
\definecolor{uuuuuu}{rgb}{0.26666666666666666,0.26666666666666666,0.26666666666666666}
\begin{tikzpicture}[line cap=round,line join=round,>=triangle 45,x=1cm,y=1cm]

\draw  (-1,1)-- (1,1); 
\draw  (1,1)-- (1,-1);
\draw  (1,-1)-- (-1,-1);
\draw  (-1,-1)-- (-1,1);
\draw  (-1,1)-- (-0.4,0);
\draw  (-0.4,0)-- (-1,-1);
\draw  (-1,-1)-- (0,-0.4);
\draw  (0,-0.4)-- (1,-1);
\draw  (1,-1)-- (0.4,0);
\draw  (0.4,0)-- (1,1);
\draw  (1,1)-- (0,0.4);
\draw  (0,0.4)-- (-1,1);
\draw  (-0.4,0)-- (0,0.4);
\draw  (0,0.4)-- (0.4,0);
\draw  (0.4,0)-- (0,-0.4);
\draw  (0,-0.4)-- (-0.4,0);
\draw  (-0.4,0)-- (0.4,0);
\draw  (0,0.4)-- (0,-0.4);

\draw  (-1+2,1)-- (1+2,1); 
\draw  (1+2,1)-- (1+2,-1);
\draw  (1+2,-1)-- (-1+2,-1);
\draw  (-1+2,-1)-- (-1+2,1);
\draw  (-1+2,1)-- (-0.4+2,0);
\draw  (-0.4+2,0)-- (-1+2,-1);
\draw  (-1+2,-1)-- (0+2,-0.4);
\draw  (0+2,-0.4)-- (1+2,-1);
\draw  (1+2,-1)-- (0.4+2,0);
\draw  (0.4+2,0)-- (1+2,1);
\draw  (1+2,1)-- (0+2,0.4);
\draw  (0+2,0.4)-- (-1+2,1);
\draw  (-0.4+2,0)-- (0+2,0.4);
\draw  (0+2,0.4)-- (0.4+2,0);
\draw  (0.4+2,0)-- (0+2,-0.4);
\draw  (0+2,-0.4)-- (-0.4+2,0);
\draw  (-0.4+2,0)-- (0.4+2,0);
\draw  (0+2,0.4)-- (0+2,-0.4);

\draw  (-1+4,1)-- (1+4,1); 
\draw  (1+4,1)-- (1+4,-1);
\draw  (1+4,-1)-- (-1+4,-1);
\draw  (-1+4,-1)-- (-1+4,1);
\draw  (-1+4,1)-- (-0.4+4,0);
\draw  (-0.4+4,0)-- (-1+4,-1);
\draw  (-1+4,-1)-- (0+4,-0.4);
\draw  (0+4,-0.4)-- (1+4,-1);
\draw  (1+4,-1)-- (0.4+4,0);
\draw  (0.4+4,0)-- (1+4,1);
\draw  (1+4,1)-- (0+4,0.4);
\draw  (0+4,0.4)-- (-1+4,1);
\draw  (-0.4+4,0)-- (0+4,0.4);
\draw  (0+4,0.4)-- (0.4+4,0);
\draw  (0.4+4,0)-- (0+4,-0.4);
\draw  (0+4,-0.4)-- (-0.4+4,0);
\draw  (-0.4+4,0)-- (0.4+4,0);
\draw  (0+4,0.4)-- (0+4,-0.4);

\draw  (-1,1+2)-- (1,1+2); 
\draw  (1,1+2)-- (1,-1+2);
\draw  (1,-1+2)-- (-1,-1+2);
\draw  (-1,-1+2)-- (-1,1+2);
\draw  (-1,1+2)-- (-0.4,0+2);
\draw  (-0.4,0+2)-- (-1,-1+2);
\draw  (-1,-1+2)-- (0,-0.4+2);
\draw  (0,-0.4+2)-- (1,-1+2);
\draw  (1,-1+2)-- (0.4,0+2);
\draw  (0.4,0+2)-- (1,1+2);
\draw  (1,1+2)-- (0,0.4+2);
\draw  (0,0.4+2)-- (-1,1+2);
\draw  (-0.4,0+2)-- (0,0.4+2);
\draw  (0,0.4+2)-- (0.4,0+2);
\draw  (0.4,0+2)-- (0,-0.4+2);
\draw  (0,-0.4+2)-- (-0.4,0+2);
\draw  (-0.4,0+2)-- (0.4,0+2);
\draw  (0,0.4+2)-- (0,-0.4+2);

\draw  (-1+2,1+2)-- (1+2,1+2); 
\draw  (1+2,1+2)-- (1+2,-1+2);
\draw  (1+2,-1+2)-- (-1+2,-1+2);
\draw  (-1+2,-1+2)-- (-1+2,1+2);
\draw  (-1+2,1+2)-- (-0.4+2,0+2);
\draw  (-0.4+2,0+2)-- (-1+2,-1+2);
\draw  (-1+2,-1+2)-- (0+2,-0.4+2);
\draw  (0+2,-0.4+2)-- (1+2,-1+2);
\draw  (1+2,-1+2)-- (0.4+2,0+2);
\draw  (0.4+2,0+2)-- (1+2,1+2);
\draw  (1+2,1+2)-- (0+2,0.4+2);
\draw  (0+2,0.4+2)-- (-1+2,1+2);
\draw  (-0.4+2,0+2)-- (0+2,0.4+2);
\draw  (0+2,0.4+2)-- (0.4+2,0+2);
\draw  (0.4+2,0+2)-- (0+2,-0.4+2);
\draw  (0+2,-0.4+2)-- (-0.4+2,0+2);
\draw  (-0.4+2,0+2)-- (0.4+2,0+2);
\draw  (0+2,0.4+2)-- (0+2,-0.4+2);

\draw  (-1+4,1+2)-- (1+4,1+2); 
\draw  (1+4,1+2)-- (1+4,-1+2);
\draw  (1+4,-1+2)-- (-1+4,-1+2);
\draw  (-1+4,-1+2)-- (-1+4,1+2);
\draw  (-1+4,1+2)-- (-0.4+4,0+2);
\draw  (-0.4+4,0+2)-- (-1+4,-1+2);
\draw  (-1+4,-1+2)-- (0+4,-0.4+2);
\draw  (0+4,-0.4+2)-- (1+4,-1+2);
\draw  (1+4,-1+2)-- (0.4+4,0+2);
\draw  (0.4+4,0+2)-- (1+4,1+2);
\draw  (1+4,1+2)-- (0+4,0.4+2);
\draw  (0+4,0.4+2)-- (-1+4,1+2);
\draw  (-0.4+4,0+2)-- (0+4,0.4+2);
\draw  (0+4,0.4+2)-- (0.4+4,0+2);
\draw  (0.4+4,0+2)-- (0+4,-0.4+2);
\draw  (0+4,-0.4+2)-- (-0.4+4,0+2);
\draw  (-0.4+4,0+2)-- (0.4+4,0+2);
\draw  (0+4,0.4+2)-- (0+4,-0.4+2);

\draw  (-1,1+4)-- (1,1+4); 
\draw  (1,1+4)-- (1,-1+4);
\draw  (1,-1+4)-- (-1,-1+4);
\draw  (-1,-1+4)-- (-1,1+4);
\draw  (-1,1+4)-- (-0.4,0+4);
\draw  (-0.4,0+4)-- (-1,-1+4);
\draw  (-1,-1+4)-- (0,-0.4+4);
\draw  (0,-0.4+4)-- (1,-1+4);
\draw  (1,-1+4)-- (0.4,0+4);
\draw  (0.4,0+4)-- (1,1+4);
\draw  (1,1+4)-- (0,0.4+4);
\draw  (0,0.4+4)-- (-1,1+4);
\draw  (-0.4,0+4)-- (0,0.4+4);
\draw  (0,0.4+4)-- (0.4,0+4);
\draw  (0.4,0+4)-- (0,-0.4+4);
\draw  (0,-0.4+4)-- (-0.4,0+4);
\draw  (-0.4,0+4)-- (0.4,0+4);
\draw  (0,0.4+4)-- (0,-0.4+4);

\draw  (-1+2,1+4)-- (1+2,1+4); 
\draw  (1+2,1+4)-- (1+2,-1+4);
\draw  (1+2,-1+4)-- (-1+2,-1+4);
\draw  (-1+2,-1+4)-- (-1+2,1+4);
\draw  (-1+2,1+4)-- (-0.4+2,0+4);
\draw  (-0.4+2,0+4)-- (-1+2,-1+4);
\draw  (-1+2,-1+4)-- (0+2,-0.4+4);
\draw  (0+2,-0.4+4)-- (1+2,-1+4);
\draw  (1+2,-1+4)-- (0.4+2,0+4);
\draw  (0.4+2,0+4)-- (1+2,1+4);
\draw  (1+2,1+4)-- (0+2,0.4+4);
\draw  (0+2,0.4+4)-- (-1+2,1+4);
\draw  (-0.4+2,0+4)-- (0+2,0.4+4);
\draw  (0+2,0.4+4)-- (0.4+2,0+4);
\draw  (0.4+2,0+4)-- (0+2,-0.4+4);
\draw  (0+2,-0.4+4)-- (-0.4+2,0+4);
\draw  (-0.4+2,0+4)-- (0.4+2,0+4);
\draw  (0+2,0.4+4)-- (0+2,-0.4+4);

\draw  (-1+4,1+4)-- (1+4,1+4); 
\draw  (1+4,1+4)-- (1+4,-1+4);
\draw  (1+4,-1+4)-- (-1+4,-1+4);
\draw  (-1+4,-1+4)-- (-1+4,1+4);
\draw  (-1+4,1+4)-- (-0.4+4,0+4);
\draw  (-0.4+4,0+4)-- (-1+4,-1+4);
\draw  (-1+4,-1+4)-- (0+4,-0.4+4);
\draw  (0+4,-0.4+4)-- (1+4,-1+4);
\draw  (1+4,-1+4)-- (0.4+4,0+4);
\draw  (0.4+4,0+4)-- (1+4,1+4);
\draw  (1+4,1+4)-- (0+4,0.4+4);
\draw  (0+4,0.4+4)-- (-1+4,1+4);
\draw  (-0.4+4,0+4)-- (0+4,0.4+4);
\draw  (0+4,0.4+4)-- (0.4+4,0+4);
\draw  (0.4+4,0+4)-- (0+4,-0.4+4);
\draw  (0+4,-0.4+4)-- (-0.4+4,0+4);
\draw  (-0.4+4,0+4)-- (0.4+4,0+4);
\draw  (0+4,0.4+4)-- (0+4,-0.4+4);

\draw (1+4,1.3+4) node {$(1,1)$};
\draw (1+4,-1.3) node {$(1,-1)$};
\draw (-1,-1.3) node {$(-1,-1)$};
\draw (-1,1.3+4) node {$(-1,1)$};
\end{tikzpicture}

%% file: supplements/one_corridor_quads.tex
\definecolor{uququq}{rgb}{0.25098039215686274,0.25098039215686274,0.25098039215686274}
\begin{tikzpicture}[line cap=round,line join=round,>=triangle 45,x=0.5cm,y=0.5cm]
\draw  (-7,1)-- (1,1); 
\draw  (1,1)-- (1,-1);
\draw  (1,-1)-- (-7,-1);
\draw  (-7,-1)-- (-7,1);
\draw  (-5,1)-- (-5,-1);
\draw  (-4.3,-1)-- (-4.3,1);
\draw  (-3,1)-- (-3,-1);
\draw  (-2,-1)-- (-2,1);
\begin{scriptsize}
\draw [fill= black] (-7,1) circle (2pt);
\draw [fill= black] (-7,-1) circle (2pt);
\draw [fill= black] (-5,1) circle (2pt);
\draw [fill= black] (-5,-1) circle (2pt);
\draw [fill= black] (-4.3,-1) circle (2pt);
\draw [fill= black] (-4.3,1) circle (2pt);
\draw [fill= black] (-3,1) circle (2pt);
\draw [fill= black] (-3,-1) circle (2pt);
\draw [fill= black] (-2,1) circle (2pt);
\draw [fill= black] (-2,-1) circle (2pt);
\draw [fill= black] (1,-1) circle (2pt);
\draw [fill= black] (1,1) circle (2pt);
\end{scriptsize}
\end{tikzpicture}

%% file: supplements/two_corridor_quads.tex
\definecolor{uququq}{rgb}{0.25098039215686274,0.25098039215686274,0.25098039215686274}
\begin{tikzpicture}[line cap=round,line join=round,>=triangle 45,x=0.5cm,y=0.5cm]
\draw (-7,1)-- (1,1);
\draw  (1,1)-- (1,-1);
\draw  (1,-1)-- (-7,-1);
\draw  (-7,-1)-- (-7,1);
\draw  (-5,1)-- (-5,-1);
\draw  (-4.3,-1)-- (-4.3,1);
\draw  (-3,1)-- (-3,-1);
\draw  (-2,-1)-- (-2,1);

\draw  (-7,-1)-- (1,-1);
\draw  (1,-1)-- (1,-3);
\draw  (1,-3)-- (-7,-3);
\draw  (-7,-3)-- (-7,-1);
\draw  (-5,-1)-- (-5,-3);
\draw  (-4.3,-3)-- (-4.3,-1);
\draw  (-3,-1)-- (-3,-3);
\draw  (-2,-3)-- (-2,-1);
\begin{scriptsize}
\draw [fill= black] (-7,1) circle (2pt);
\draw [fill= black] (-7,-1) circle (2pt);
\draw [fill= black] (-5,1) circle (2pt);
\draw [fill= black] (-5,-1) circle (2pt);
\draw [fill= black] (-4.3,-1) circle (2pt);
\draw [fill= black] (-4.3,1) circle (2pt);
\draw [fill= black] (-3,1) circle (2pt);
\draw [fill= black] (-3,-1) circle (2pt);
\draw [fill= black] (-2,1) circle (2pt);
\draw [fill= black] (-2,-1) circle (2pt);
\draw [fill= black] (1,-1) circle (2pt);
\draw [fill= black] (1,1) circle (2pt);

\draw [fill= black] (-7,-1) circle (2pt);
\draw [fill= black] (-7,-3) circle (2pt);
\draw [fill= black] (-5,-1) circle (2pt);
\draw [fill= black] (-5,-3) circle (2pt);
\draw [fill= black] (-4.3,-3) circle (2pt);
\draw [fill= black] (-4.3,-1) circle (2pt);
\draw [fill= black] (-3,-1) circle (2pt);
\draw [fill= black] (-3,-3) circle (2pt);
\draw [fill= black] (-2,-1) circle (2pt);
\draw [fill= black] (-2,-3) circle (2pt);
\draw [fill= black] (1,-3) circle (2pt);
\draw [fill= black] (1,-1) circle (2pt);
\end{scriptsize}
\end{tikzpicture}

%% file: supplements/tikz_algo1.tex
\begin{tabular}[t]{l l l l l}
&
\tabcell{\begin{tikzpicture}
	\phantom{	\draw (0.3,2.3) node[right] 
	{$\operatorname{deg}(z^\prime,\Nout) = 
			1 $} ;}
  \draw
    (0, 0) grid[step=1cm] (3, 2)
    (0, 1) -- (1, 2)
    (0, 0) -- (2, 2)
    (1, 0) -- (3, 2)
    (2, 0) -- (3, 1)
    (3, 0) -- (3.3,0)
    (3, 1) -- (3.3,1)
    (3, 2) -- (3.3,2)
    (0, 2) -- (0, 2.3)
    (1, 2) -- (1, 2.3)
    (2, 2) -- (2, 2.3)
    (3, 2) -- (3, 2.3)
  ;  
  \draw (0,-0.3) node  {$z^\prime$};
  \draw (1.3,0.8) node  {$z$};
  \draw[dashed,line width=2pt] (0,0) -- (1,1);
  \draw (0.2,-0.3) node[right] {$\operatorname{deg}(z^\prime,\Nout) = 1 $} ;
  \fill[radius=3pt, fill=orange]
    \foreach \y in {0, ..., 2} {
      (0, \y) circle[]
    }
  ;
  \fill[radius=3pt, fill=orange]
    \foreach \x in {1, ..., 3} {
       (\x, 0) circle[]
    }
  ;
\node[regular polygon, regular polygon sides=3, inner sep=1.3pt, draw = red, fill = red] at (1,1) {};
\node[regular polygon, regular polygon sides=3, inner sep=1.3pt, draw = red, fill = red] at (2,1) {};
\node[regular polygon, regular polygon sides=3, inner sep=1.3pt, draw = red, fill = red] at (3,1) {};
 
\foreach \x in {1, ..., 3} {
  \node[regular polygon, regular polygon sides=3, inner sep=1.3pt, draw = red, fill = red] at (\x, 2) {} ;
    }

\end{tikzpicture}}\hspace{1cm}
& \hspace{-30pt} \( \xrightarrow[\shortstack{$\textcolor{black}{\Nin}= 
\textcolor{black}{\Nin}\setminus \{z\}$}]
{ \shortstack{
$\textcolor{black}{\Nout}=\textcolor{black}{\Nout}\cup \{z\}$ } }\) 
& \hspace{-30pt}
\tabcell{\begin{tikzpicture}
\phantom{	\draw (1.3,-0.3) node[right] {$\operatorname{deg}(z^\prime,\Nout) = 
1 $} ;}
\phantom{	\draw (-0.7,2.3) node[right] {$\operatorname{deg}(z^\prime,\Nout) = 
		1 $} ;}
  \draw
    (0, 0) grid[step=1cm] (3, 2)
    (0, 1) -- (1, 2)
    (0, 0) -- (2, 2)
    (1, 0) -- (3, 2)
    (2, 0) -- (3, 1)
    (3, 0) -- (3.3,0)
    (3, 1) -- (3.3,1)
    (3, 2) -- (3.3,2)
    (0, 2) -- (0, 2.3)
    (1, 2) -- (1, 2.3)
    (2, 2) -- (2, 2.3)
    (3, 2) -- (3, 2.3)
  ;
  
  \draw (0,-0.3) node  {$z^\prime$};
  \draw (1.3,0.8) node  {$z$};

  \fill[radius=3pt, fill=orange]
    \foreach \y in {0, ..., 2} {
      (0, \y) circle[]
    }
  ;
  \fill[radius=3pt, fill=orange]
    \foreach \x in {1, ..., 3} {
       (\x, 0) circle[]
    }
  ;
  \fill[radius=3pt, fill=orange]
    \foreach \x in {1, ..., 1} {
       (\x, 1) circle[]
    }
  ;

    \foreach \x in {2, ..., 3} {
      \node[regular polygon, regular polygon sides=3, inner sep=1.3pt, draw = red, fill = red] at (\x, 1) {};
    }
    \foreach \x in {1, ..., 3} {
      \node[regular polygon, regular polygon sides=3, inner sep=1.3pt, draw = red, fill = red] at (\x, 2) {};
    }
\end{tikzpicture}}&\(\phantom{\xrightarrow{loooong}}\)\hspace{-20pt}\\
\(\xrightarrow{\phantom{loooong}}\)\hspace{-20pt} & 
\tabcell{\begin{tikzpicture}
	\phantom{\draw (0,-0.3) node  {$z^\prime$};}
	\phantom{	\draw (0.3,2.3) node[right] 
	{$\operatorname{deg}(z^\prime,\Nout) = 
			1 $} ;}
  \draw
    (0, 0) grid[step=1cm] (3, 2)
    (0, 1) -- (1, 2)
    (0, 0) -- (2, 2)
    (1, 0) -- (3, 2)
    (2, 0) -- (3, 1)
    (3, 0) -- (3.3,0)
    (3, 1) -- (3.3,1)
    (3, 2) -- (3.3,2)
    (0, 2) -- (0, 2.3)
    (1, 2) -- (1, 2.3)
    (2, 2) -- (2, 2.3)
    (3, 2) -- (3, 2.3)
    ;
  \draw[dashed,line width=2pt] (1,0) -- (2,1);
  \draw (1.3,-0.3) node[right] {$\operatorname{deg}(z^\prime,\Nout) = 1 $} ;
  \draw (1,-0.3) node  {$z^\prime$};
  \draw (2.3,0.8) node  {$z$};
  \fill[radius=3pt, fill=orange]
    \foreach \y in {1, ..., 2} {
      (0, \y) circle[]
    }
  ;
  \fill[radius=3pt, fill=orange]
    \foreach \x in {0, ..., 3} {
       (\x, 0) circle[]
    }
  ;
  \fill[radius=3pt, fill=orange]
    \foreach \x in {1, ..., 1} {
       (\x, 1) circle[]
    }
  ;

    \foreach \x in {2, ..., 3} {
       \node[regular polygon, regular polygon sides=3, inner sep=1.3pt, draw = red, fill = red] at (\x, 1) {};
    }
    \foreach \x in {1, ..., 3} {
       \node[regular polygon, regular polygon sides=3, inner sep=1.3pt, draw = red, fill = red] at (\x, 2) {};
    }
  ;  
\end{tikzpicture}}
& \hspace{-30pt}\(\xrightarrow[\shortstack{$\textcolor{black}{\Nin}= 
\textcolor{black}{\Nin}\setminus \{z\}$}]
{ \shortstack{
$\textcolor{black}{\Nout}=\textcolor{black}{\Nout}\cup \{z\}$ } }\) 
& \hspace{-30pt}
\tabcell{\begin{tikzpicture}
	\phantom{\draw (0,-0.3) node  {$z^\prime$};}
	\phantom{\draw (1.3,-0.3) node[right] {$\operatorname{deg}(z^\prime,\Nout) 
	= 1 $} ;}
\phantom{	\draw (-0.7,2.3) node[right] {$\operatorname{deg}(z^\prime,\Nout) = 
		1 $} ;}
  \draw
    (0, 0) grid[step=1cm] (3, 2)
    (0, 1) -- (1, 2)
    (0, 0) -- (2, 2)
    (1, 0) -- (3, 2)
    (2, 0) -- (3, 1)
    (3, 0) -- (3.3,0)
    (3, 1) -- (3.3,1)
    (3, 2) -- (3.3,2)
    (0, 2) -- (0, 2.3)
    (1, 2) -- (1, 2.3)
    (2, 2) -- (2, 2.3)
    (3, 2) -- (3, 2.3)
    ;
  \fill[radius=3pt, fill=orange]
    \foreach \y in {1, ..., 2} {
      (0, \y) circle[]
    }
  ;
  \fill[radius=3pt, fill=orange]
    \foreach \x in {0, ..., 3} {
       (\x, 0) circle[]
    }
  ;
  \fill[radius=3pt, fill=orange]
    \foreach \x in {1, ..., 2} {
       (\x, 1) circle[]
    }
  ;
   \foreach \x in {3, ..., 3} {
       \node[regular polygon, regular polygon sides=3, inner sep=1.3pt, draw = red, fill = red] at (\x, 1){};
    }
    \foreach \x in {1, ..., 3} {
      \node[regular polygon, regular polygon sides=3, inner sep=1.3pt, draw = red, fill = red] at (\x, 2) {};
    }
\end{tikzpicture}}
\end{tabular}